\documentclass[11pt, a4paper, english]{amsart}
\usepackage[top=2cm, bottom=2cm, left=2.5cm, right=2.5cm, twoside=false]{geometry}
\usepackage[english]{babel}
\usepackage[utf8]{inputenc}
\usepackage{amstext}
\usepackage{ifpdf}
\usepackage{enumerate}
\usepackage{amscd,amssymb,amsthm,amsmath,amssymb,textcomp}
\usepackage[matrix,arrow,curve]{xy}
\usepackage{pdflscape}
\usepackage{mathrsfs, setspace, fancyhdr, bm}

\newtheorem{theorem}[equation]{Theorem}

\newtheorem{proposition}[equation]{Proposition}
\newtheorem{lemma}[equation]{Lemma}
\newtheorem{corollary}[equation]{Corollary}
\newtheorem{conjecture}[equation]{Conjecture}

\newtheorem{exercise}[equation]{Exercise}

\theoremstyle{definition}
\newtheorem{example}[equation]{Example}
\newtheorem{definition}[equation]{Definition}
\newtheorem{notation}[equation]{Notation}
\theoremstyle{remark}
\newtheorem{remark}[equation]{Remark}

\makeatletter\@addtoreset{equation}{section} \makeatother

\newcommand{\mumu}{\boldsymbol{\mu}}

\newcommand\tone[1]{\hat{#1}}
\newcommand\ttwo[1]{\bar{#1}}

\title{Toric $G$-solid Fano threefolds}

\author{Ivan Cheltsov, Adrien Dubouloz, Takashi Kishimoto}

\address{\emph{Ivan Cheltsov}
\newline
\textnormal{University of Edinburgh,  Edinburgh, Scotland}
\newline
\textnormal{\texttt{I.Cheltsov@ed.ac.uk}}}

\address{\emph{Adrien Dubouloz}
\newline
\textnormal{Universit\'e de Bourgogne, Dijon, France}
\newline
\textnormal{\texttt{adrien.dubouloz@u-bourgogne.fr}}}

\address{\emph{Takashi Kishimoto}
\newline
\textnormal{Saitama University, Saitama, Japan}
\newline
\textnormal{\texttt{kisimoto.takasi@gmail.com}}}

\begin{document}

\begin{abstract}
We study toric $G$-solid Fano threefolds that have at most terminal singularities,
where $G$ is an algebraic subgroup of the normalizer of a maximal torus in their automorphism groups.
\end{abstract}

\maketitle

All varieties are assumed to be projective and defined over the field of complex numbers.

\section{Introduction}
\label{section:intro}

Fano varieties with many symmetries appear naturally in several geometric problems.
A special role among them is played by the so-called $G$-Fano varieties \cite{Prokhorov2013},
which naturally occur as the end product of the equivariant Minimal Model~Program for rationally connected varieties.
Recall from \cite{Prokhorov2013} that a \emph{$G$-Fano variety} is a pair $(X,G)$ consisting of a Fano variety $X$
and an algebraic subgroup $G$ in $\mathrm{Aut}(X)$ such that
\begin{enumerate}
\item the singularities of $X$ are terminal (mild);
\item the $G$-invariant part $\mathrm{Cl}(X)^{G}$ of the class group of $X$ has rank $1$ ($G$-minimal).
\end{enumerate}

In dimension two, we know the complete list of $G$-Fano varieties \cite{DolgachevIskovskikh}, which are traditionally called $G$-del Pezzo surfaces.
In \cite{Prokhorov,Prokhorov2013}, Prokhorov obtained many deep results about $G$-Fano threefolds for G finite. A complete classification is still lacking. In higher dimensions, our knowledge of $G$-Fano varieties is limited to some sporadic examples.

Since by definition a $G$-Fano variety $X$ is a $G$-Mori fibre space (see \cite[Definition~1.1.5]{CheltsovShramov}),
to describe its $G$-equivariant birational geometry, it is enough to classify all $G$-birational maps from $X$ to other $G$-Mori fibre spaces.
By \cite{Co95,HaconMcKernan}, each such birational map can be decomposed into a sequence of elementary links, which are known as $G$-Sarkisov links.
Then, following \cite{AO2018,CheltsovShramov,CheltsovShramov2017,Co00}, we say that a $G$-Fano variety $X$ is:
\begin{itemize}
\item $G$-birationally super-rigid if no $G$-Sarkisov link starts at $X$;
\item $G$-birationally rigid if every $G$-Sarkisov link that starts at $X$ also ends at $X$;
\item $G$-solid if $X$ is not $G$-birational to a~$G$-Mori fibre space with positive dimensional base.
\end{itemize}
If $X$ is $G$-solid, then all $G$-Mori fibre spaces that are $G$-birational to $X$ are terminal Fano threefolds ---
they form a set $\mathcal{P}_G(X)$, which we call the \emph{$G$-pliability} of $X$ \cite{CortiMella}.
For instance, if $X$ is $G$-solid, then $\mathcal{P}_G(X)=\{X\}$ if and only if $X$ is $G$-birationally rigid.

In this paper, we consider toric $G$-Fano varieties in the case where $G$ is an algebraic subgroup in $\mathrm{Aut}(X)$
that normalizes a maximal torus $\mathbb{T}\cong\mathbb{G}_m^n$, where $n=\mathrm{dim}(X)$.
In~this case, letting $G_X$ be the normalizer of the torus $\mathbb{T}$ in $\mathrm{Aut}(X)$, we have a split exact sequence of groups
$$
\xymatrix{
1\ar@{->}[r]&\mathbb{T}\ar@{->}[r]& G_{X}\ar@{->}[r]^{\nu_X}&\mathbb{W}_X\ar@{->}[r]& 1,}
$$
where $\mathbb{W}_X$ is a finite subgroup of $\mathrm{GL}_n(\mathbb{Z})$, known as the Weyl group. It is a natural problem to determine for such groups $G$ which $G$-Fano varieties are $G$-solid, and to characterize which of these are $G$-birationally rigid or super-rigid.
For surfaces, it is easy to give a satisfactory answer to these questions.

\begin{exercise}[{cf. \cite{DolgachevIskovskikh,Iskovskikh,Sakovics,Walter}}] \label{exercise:toric-surfaces}
Let $X$ be a smooth toric del Pezzo surface and let $G$ be a subgroup of $G_X$. If $X$ is $G$-minimal and $G$-solid, then one of the following cases holds:
\begin{itemize}
\item[(i)] $X=\mathbb{P}^2$, $\mathbb{W}_X\cong\mathfrak{S}_3$ and $\nu_X(G)$ contains the subgroup isomorphic to $\mumu_3$;
\item[(ii)] $X=\mathbb{P}^1\times\mathbb{P}^1$, $\mathbb{W}_X\cong\mathrm{D}_{8}$ and $\nu_X(G)$ contains the subgroup isomorphic to $\mumu_4$;
\item[(iii)] $X$ is the del Pezzo surface of~degree~$6$, $\mathbb{W}_X\cong\mathfrak{S}_3\times\mumu_2$, and either $\nu_X(G)$ contains the subgroup isomorphic to $\mumu_6$ or it contains the  subgroup isomorphic to $\mathfrak{S}_3$ that acts transitively on the $(-1)$-curves in $X$.
\end{itemize}
In each of these three cases, $X$ is $G$-minimal and $G$-birationally rigid provided that $|G|\geqslant 72$.
\end{exercise}

In this paper, we obtain a similar answer for three-dimensional toric Fano varieties.
To~state~it, let $V_6=\mathbb{P}^1\times\mathbb{P}^1\times\mathbb{P}^1$,
let $V_4$ be the toric complete intersection in $\mathbb{P}^5$ given by
$$
xu-yw=xu-zt=0,
$$
let $Y_{24}$ be the toric divisor of degree $(1,1,1,1)$ in $\mathbb{P}^1\times\mathbb{P}^1\times\mathbb{P}^1\times\mathbb{P}^1$ given by
$$
x_1x_2x_3x_4-y_1y_2y_3y_4=0,
$$
and let $X_{24}$ be the toric Fano threefold \textnumero{47} in \cite{grdb} (see Section~\ref{subsection:Fano-Enriques} for its construction).
The Weyl groups $\mathbb{W}_X$ of these toric Fano threefolds are all isomorphic to the group $\mathfrak{S}_4\times\mumu_2$, and we have the following result:

\begin{theorem}
\label{theorem:main}
Let $X$ be a toric Fano threefold that have at most terminal singularities,
let $\mathbb{T}$ be a maximal torus in $\mathrm{Aut}(X)$ and let $G_X$ be its normalizer in $\mathrm{Aut}(X)$.
Then the following two conditions are equivalent:
\begin{enumerate}
\item[(i)] $X$ is $G_X$-minimal and $G_X$-solid;
\item[(ii)] $X$  is one of the threefolds $V_6$, $V_{4}$, $X_{24}$, $Y_{24}$ and $\mathbb{P}^3$.
\end{enumerate}
Let $G$ be an algebraic subgroup in $G_X$ and let $\nu_X\colon G_X\to\mathbb{W}_X=G_X/\mathbb{T}$ be the quotient homomorphism.
If $X$ is one of the toric Fano threefolds $V_6$, $V_{4}$, $X_{24}$, $Y_{24}$ and $\mathbb{P}^3$, then the following assertions hold:
\begin{enumerate}
\item if $X$ is $G$-minimal and $G$-solid, then $\nu_X(G)$ contains a subgroup isomorphic to $\mathfrak{A}_4$;
\item if $\nu_X(G)$ contains a subgroup isomorphic to $\mathfrak{A}_4$, then $X$ is $G$-minimal unless
\begin{enumerate}
\item $X=V_4$, $\nu_X(G)\cong\mathfrak{S}_4$ and $G$ acts intransitively on $\mathbb{T}$-invariant surfaces;
\item $X=V_4$ and $\nu_X(G)\cong\mathfrak{A}_4$.
\end{enumerate}
\item if $X$ is $G$-minimal, $\nu_X(G)$ contains a subgroup isomorphic to $\mathfrak{A}_4$ and $|G|\geqslant 32\cdot 24^4$, then $X$ is $G$-solid.
\end{enumerate}
\end{theorem}

If $X=V_4$ and $\nu_X(G)$ contains a subgroup isomorphic to $\mathfrak{A}_4$, then $X$ is not $G$-minimal
if and only if there exists the following $G$-commutative diagram:
$$
\xymatrix{
&\widetilde{V}_4\ar@{->}[ld]_{\alpha}\ar@{->}[rd]^{\beta}&\\%
V_4&&\mathbb{P}^3\ar@{-->}[ll]}
$$
where $\beta$ is the blow-up of the four $\mathbb{T}$-invariant points,
$\alpha$ is the contraction of the proper transforms of the six $\mathbb{T}$-invariant lines,
and the dashed arrow is the birational map that is given by the linear system of quadric surfaces that pass through the four $\mathbb{T}$-invariant points.

\begin{remark}
\label{remark:G-X-Aut}
If $X$ is one of the toric Fano threefolds  $V_{4}$, $X_{24}$ or $Y_{24}$, then $G_X=\mathrm{Aut}(X)$.
\end{remark}

If $X$ is one of the toric Fano threefolds $V_6$, $V_{4}$, $X_{24}$, $Y_{24}$,~$\mathbb{P}^3$,
and $G$ is an algebraic subgroup in $\mathrm{Aut}(X)$ such that the threefold $X$ is $G$-minimal,
$\nu_X(G)$ contains a subgroup isomorphic~to~$\mathfrak{A}_4$, and $|G|\geqslant 32\cdot 24^4$,
then the threefold $X$ is $G$-solid by Theorem~\ref{theorem:main}.
In~this case, we describe all (possible) $G$-birational maps between these Fano threefolds.
We summarize this description in the table presented in Appendix~\ref{section:table}.
It gives

\begin{corollary}[{cf. \cite{Avilov,CheltsovShramov2017}}]
\label{corollary:main}
Let $X$ be one of the toric Fano threefolds $V_6$, $V_{4}$, $X_{24}$, $Y_{24}$,~$\mathbb{P}^3$,
let~$\mathbb{T}$~be~a~maximal torus in $\mathrm{Aut}(X)$ and let $G_X$ be its normalizer in $\mathrm{Aut}(X)$.
Then the following three conditions are equivalent:
\begin{enumerate}
\item[(i)] $X$ is $G_X$-minimal and $G_X$-birationally rigid;
\item[(ii)] $X$ is $G_X$-minimal and $G_X$-birationally super-rigid;
\item[(iii)] $X$ is isomorphic to either $V_6$ or $Y_{24}$.
\end{enumerate}
Let $G$ be an algebraic subgroup in $G_X$ and let $\nu_X\colon G_X\to\mathbb{W}_X=G_X/\mathbb{T}$ be the quotient homomorphism.
Assume that $\nu_X(G)$ contains a subgroup isomorphic to $\mathfrak{A}_4$. Then the following assertions hold:
\begin{enumerate}
\item if $X$ is $G$-minimal and $G$-birationally rigid, then $X=V_6$ or $X=Y_{24}$;
\item if $X=V_6$ or $X=Y_{24}$, and $|G|\geqslant 32\cdot 24^4$, then $X$ is $G$-birationally super-rigid.
\end{enumerate}
\end{corollary}

In Section~\ref{section:toric}, we provide a criterion for a $G$-minimal toric Fano variety $X$ to be $G$-solid
in the case where $G$ is an algebraic subgroup of $G_X$ that contains the maximal~torus $\mathbb{T}$.
Unfortunately, we do no know how to generalize this criterion for finite subgroups in~$G_X$.
Nevertheless, Exercise~\ref{exercise:toric-surfaces} and Theorem~\ref{theorem:main} suggest the following conjecture.

\begin{conjecture}
\label{conjecture}
Let $X$ be a toric Fano variety with at most terminal singularities
and let $G$ be a subgroup in $G_X$ that contains $\mathbb{T}$ such that $X$ is $G$-minimal and $G$-solid.
Then there exists a constant $c_X>0$ such that for every  finite subgroup $H\subset G$ such that $\nu_X(H)=\nu_X(G)$,
the Fano variety $X$ is $H$-solid provided that $|H|\geqslant c_X$.
\end{conjecture}

The structure of the article is the following.
In Section~\ref{section:preliminaries}, we present results that will be used in the proof of Theorem~\ref{theorem:main}.
In Sections~\ref{section:toric} and \ref{section:Fano-threefolds}, we prove Theorem~\ref{theorem:main} for infinite algebraic groups
using toric geometry.
In Section~\ref{section:two-Sarkisov-links}, we explicitly describe two (known) equivariant toric Sarkisov links that start at $X_{24}$.
In Section~\ref{section:proof}, we give an~alternative proof of Theorem~\ref{theorem:main}(3) for infinite algebraic groups,
which can be generalized for large finite groups.
In Section~\ref{section:approximation}, we complete the proof of Theorem~\ref{theorem:main} by proving its part (3) for finite groups
(the remaining assertions of Theorem~\ref{theorem:main} for finite groups follows from the results proven in Sections~\ref{section:toric} and \ref{section:Fano-threefolds}).

\bigskip

\noindent
\textbf{Acknowledgements.}
This research was initiated during a research stay at Nemuro city in Hokkaido and continued at the occasion of the Kinosaki Algebraic Geometry Symposium 2019. We would like to thank the Nemuro city council for providing us with excellent working conditions. We are very grateful to Yuri Prokhorov, Hendrick S\"uss, Yuri Tschinkel and Yuri Zarhin for useful discussions and helpful comments.

The first author was supported by Leverhulme Trust Grant RPG-2021-229. The second author was partially supported by the French ANR projects ISITE-BFC "Motivic invariants of Algebraic Varieties" ANR-15-IDEX-0008 and "FIBALGA" ANR-18-CE40-0003. The third author was partially funded by Grant-in-Aid for Scientific Research of JSPS No.19K03395.

\section{Preliminary results}
\label{section:preliminaries}

In this section, we review results and notions that will be used in the proof of Theorem~\ref{theorem:main}.

\subsection{Varieties with regular group actions}
Let $G$ and $G^\prime$ be two algebraic groups,
and let $X$ and $X^\prime$ be two algebraic varieties endowed with regular
actions $m_{X}\colon G\times X\rightarrow X$ and $m_{X^\prime}\colon G^\prime\times X^\prime\rightarrow X^\prime$ of the groups $G$ and $G^\prime$ respectively.

\begin{definition}
\label{definition:equivariant}
An equivariant rational map between the varieties $X$ and $X^\prime$ is a pair consisting
of morphism of algebraic groups $\varphi\colon G\rightarrow G^\prime$ and a rational
map $\Phi\colon X\dashrightarrow X^\prime$ such that the following diagram of
rational maps commutes
$$
\xymatrix{G\times X\ar@{-->}[d]_{\varphi\times\Phi}\ar[rr]^{m_{X}} && X\ar@{-->}[d]^{\Phi}\\
G^\prime\times X^\prime\ar[rr]^{{\quad m_{X^\prime}}} && X^\prime.
}
$$
We say that the rational map $\Phi$ is $\varphi$-equivariant.
\end{definition}

If the morphism $\varphi$ in Definition~\ref{definition:equivariant} is an isomorphism and the map $\Phi$ is birational,
then letting $\rho\colon G\rightarrow\mathrm{Aut}(X)$ and $\rho^\prime\colon G^\prime\rightarrow\mathrm{Aut}(X^\prime)$
be the group homomorphisms determined by $m_{X}$ and $m_{X^\prime}$ respectively,
the commutativity of the diagram in Definition~\ref{definition:equivariant} is equivalent to the property
that
$$
\rho^\prime\circ\varphi(g)=\Phi\rho(g)\Phi^{-1}
$$
for every $g\in G$. In this paper, we are mostly interested in the case when $G\cong G^\prime$.
Because of this, we will assume in the following that $G=G^\prime$, so that both
varieties $X$ and $X^\prime$ are endowed with regular actions of the same group $G$.

\begin{definition}
\label{definition:G-map}
A $G$-equivariant rational map $X\dasharrow X^\prime$ is an $\mathrm{id}_G$-equivariant rational map $\Phi\colon X\dashrightarrow X^\prime$.
A rational $G$-map $X\dasharrow X^\prime$ is a rational map $\Phi\colon X\dashrightarrow X^\prime$
which is $\varphi$-equivariant for some automorphism $\varphi$ of $G$.
\end{definition}

We denote by $\mathrm{Bir}^{G}(X,X^\prime)$ the set of birational $G$-maps between $X$ and $X^\prime$,
and we denote by $\mathrm{Bir}_G(X,X^\prime)$ its subset consisting of $G$-equivariant birational maps $X\dasharrow X^\prime$.
If~$X=X^\prime$ then these sets are groups (with respect to composition of birational maps), which we denote by $\mathrm{Bir}^{G}(X)$ and $\mathrm{Bir}_G(X)$, respectively.

As an illustration, we describe an equivariant version of a birational map of threefolds which appeared in \cite[Proposition~2.1]{CheltsovShramov2017}.
Its nature is local, but we will use global language for simplicity of exposition.

\begin{example}
\label{ex:symbolic-BU} Let $X$ be a smooth threefold, let $P$ be a point in $X$,
let $G$ be an algebraic group that acts faithfully on $X$,
and let~$C$ be a $G$-irreducible curve in $X$ consisting  of three irreducible components
$C_1$, $C_2$ and $C_3$ meeting at a unique point $P$ and such that
the curves $C_1$, $C_2$ and $C_3$ are smooth at $P$,
and their tangent directions at $P$ generate the tangent space $T_{P}(X)$.
Let $\alpha\colon\widetilde{X}\to X$ be the blow-up of the point $P$, and let $E$ be its exceptional surface.
Denote by $\widetilde{C}_i$ the proper transform of the curve~$C_i$ on the threefold~$\widetilde{X}$.
Let $L_{ij}$ be the line in $E\cong\mathbb{P}^2$ that pass through the points $\widetilde{C}_i\cap E$ and $\widetilde{C}_j\cap E$.
Then there exists the following commutative diagram of birational $G$-maps:
$$
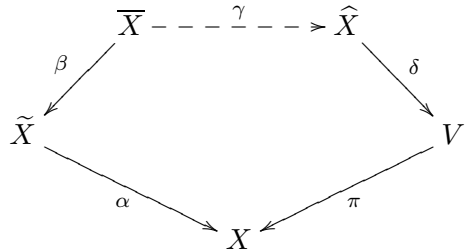
\xymatrix{
&\overline{X}\ar@{->}[ld]_{\beta}\ar@{-->}[rr]^{\gamma}&&\widehat{X}\ar@{->}[dr]^{\delta}&\\%
\widetilde{X}\ar@{->}[drr]_{\alpha}&&&&V\ar@{->}[dll]^{\pi}\\%
&&X&&}
$$
where  $\beta$ is the blow-up of the curves $\widetilde{C}_1$, $\widetilde{C}_2$ and $\widetilde{C}_3$,
the map $\gamma$ is a composition of Atiyah flops of the proper transforms on $\overline{X}$ of the curves $L_{12}$, $L_{13}$ and $L_{23}$,
and $\delta$ is the birational contraction of the proper transform of the surface $E$ to a singular point of type $\frac{1}{2}(1,1,1)$.
The morphism $\pi$ is a $G$-equivariant extremal divisorial contraction.
\end{example}

We will use the following result, which is a consequence of Luna's \'etale slice theorem (see e.\,g.\cite[Section 2.1]{Akhiezer95} and  \cite[p.~98]{Cartan}).

\begin{lemma}
\label{lemma:stabilizer-faithful} Let $G$ be a reductive group acting faithfully on a variety $X$ and let $P\in X$ be a smooth point which is fixed by the action of $G$. Then the induced linear action of $G$ on the Zariski tangent space $T_{P}(X)$ is faithful.
\end{lemma}

In the case of curves, we have the following more precise consequence:

\begin{corollary}
\label{corollary:stabilizer-faithful}
A finite group of automorphisms of a curve fixing a smooth point is a cyclic group.
\end{corollary}

\begin{corollary}
\label{lemma:curve-stabilizer-faithful}
Let $X$ be an algebraic variety with a faithful action of the group $G=\mumu_n^2$ fixing a smooth point $P\in X$.
Let $C$ be a $G$-invariant curve in $X$ containing $P$ and assume that the stabilizer in $G$ of every irreducible component of $C$ passing through $P$ acts on this component faithfully. Then $\mathrm{mult}_{P}(C)\geqslant n$.
\end{corollary}

\begin{proof}
Let $f\colon\widetilde{C}\to C$ be the normalization of the curve $C$.
The action of the group $G$ on $C$ lifts to an action on $\widetilde{C}$ preserving the preimage $F=f^{-1}(P)$ of the~point~$P$.
Let $Q$ be a point in $F$, and let $G_Q$ be its stabilizer in $G$.
Note that $Q$ is contained in a~unique irreducible component of the smooth curve $\widetilde{C}$,
which then must be $G_Q$-invariant.
Since~the group $G_Q$~acts faithfully on this component, we conclude that $G_Q$ is a cyclic subgroup of the group $G\cong\mumu_n^2$ by Corollary~\ref{corollary:stabilizer-faithful}.
Then $|G_Q|\leqslant n$, so that
$$
\mathrm{mult}_{P}(C)\geqslant |F|\geqslant\frac{|G|}{|G_Q|}\geqslant n
$$
as required.
\end{proof}

\subsection{Singularities of log-pairs}
\label{section:pairs}
Let $X$ be a threefold with at most terminal singularities, let $\mathcal{M}_X$ be a non-empty mobile linear system on $X$ that consists of $\mathbb{Q}$-Cartier divisors and let $\lambda$ be a positive rational number.

\begin{lemma}[{\cite[Exercise~6.18]{CoKoSm03}}]
\label{lemma:exercise}
Let $C$ be an irreducible curve in $X$. Assume that
$$
\mathrm{mult}_{C}\big(\mathcal{M}_X\big)\leqslant\frac{1}{\lambda}.
$$
Then $C$ is not a center of non-canonical singularities of the log pair $(X,\lambda\mathcal{M}_X)$.
\end{lemma}

The following result is due to Alessio Corti.

\begin{lemma}[{\cite[Theorem~3.1]{Co00}}]
\label{lemma:Corti-surface}
Let $C$ be an irreducible curve in $X$. Assume that $C$ is a center of non-log canonical singularities of the log pair $(X,\lambda\mathcal{M}_X)$.
Then
$$
\mathrm{mult}_{C}\Big(M_1\cdot M_2\Big)>\frac{4}{\lambda^2}
$$
for any two general surfaces $M_1$ and $M_2$ in the linear system $\mathcal{M}_X$.
\end{lemma}

The following result is due to Alexander Pukhlikov.

\begin{lemma}[{\cite{Pukhlikov}, see also \cite[Corollary 3.4]{Co00}}]
\label{lemma:Pukhlikov}
Let $P$ be a smooth point of $X$. Assume that $P$ is a center of non-canonical singularities of the log pair $(X,\lambda\mathcal{M}_X)$.
Then
$$
\mathrm{mult}_{P}\Big(M_1\cdot M_2\Big)>\frac{4}{\lambda^2}
$$
for any two general surfaces $M_1$ and $M_2$ in the linear system $\mathcal{M}_X$.
\end{lemma}

We will also need the following two results due to Kawamata \cite{Kawamata} and Corti, respectively.

\begin{lemma}
\label{lemma:Kawamata}
Let $P$ be a singular point of $X$ of type $\frac{1}{2}(1,1,1)$, let~$\pi\colon V\to X$ be the Kawamata blow-up of $P$,
let $E$ be the exceptional surface, let~$\mathcal{M}_{V}$ be the proper transform of the linear system $\mathcal{M}_X$ via $\pi$,
and let $m\in\mathbb{Q}$ be such that
$$
\mathcal{M}_V\sim_{\mathbb{Q}}\pi^*\big(\mathcal{M}_X\big)-mE.
$$
If $(X,\lambda\mathcal{M}_X)$ is not canonical at $P$ then $m>\frac{1}{2\lambda}$.
\end{lemma}

\begin{proof}
Suppose that $m\leqslant\frac{1}{2\lambda}$.
Let us seek for a contradiction.
Since
$$
K_V+\lambda\mathcal{M}_V+\Big(\lambda m-\frac{1}{2}\Big)E\sim_{\mathbb{Q}}\pi^*\big(K_X+\lambda\mathcal{M}_X\big),
$$
the pair $(V,\lambda\mathcal{M}_V)$ is not canonical at some point $O\in E$.
Then $\mathrm{mult}_O(\mathcal{M}_V)>\frac{1}{\lambda}$, so that
$$
\mathrm{mult}_O\Big(\mathcal{M}_V\vert_{E}\Big)>\frac{1}{\lambda},
$$
which is impossible, since $\mathcal{M}_V\vert_{E}\sim_{\mathbb{Q}}2mL$, where $L$ is a line in $E\cong\mathbb{P}^2$.
\end{proof}

\begin{lemma}
\label{lemma:Corti}
Let $P$ be an ordinary double point of $X$, let~$\pi\colon V\to X$ be the~blow-up of $P$, let $E$ be the exceptional surface of $\pi$,
let $\mathcal{M}_{V}$ be the~proper transform~of the linear system $\mathcal{M}_X$ via $\pi$,
and let $m\in\mathbb{Q}$ be such that
$$
\mathcal{M}_V\sim_{\mathbb{Q}}\pi^*\big(\mathcal{M}_X\big)-mE.
$$
If $P$ is a center of non-canonical singularities of the log pair $(X,\lambda\mathcal{M}_X)$ then $m>\frac{1}{\lambda}$.
\end{lemma}

\begin{proof}
This is \cite[Theorem~1.7.20]{CheltsovUMN}, which is equivalent to \cite[Theorem~3.10]{Co00}.
\end{proof}

Finally, we will need one local result proved in \cite{ACPS}.
To state it, we suppose that the threefold $X$ is endowed with an action of an algebraic group $G$,
and $\mathcal{M}_X$ is~$G$-invariant.

\begin{lemma}[{\cite[Lemma~2.4]{ACPS}}]
\label{lemma:mult-2}
Suppose that the group $G$ fixes a smooth point $P\in X$ and that its induced linear action on the Zariski tangent space $T_{P}X$ is an irreducible representation.
If $P$ is a non-canonical center of the log pair $(X,\lambda\mathcal{M}_X)$ then $\mathrm{mult}_{P}(\mathcal{M}_X)>\frac{2}{\lambda}$.
\end{lemma}

\subsection{Finite groups acting on toric varieties}
\label{subsection:finite-groups-toric}
Let $\mathbb{T}$ be a torus of dimension $d\geqslant 2$,
and let $\Gamma_n$ be a subgroup of $\mathbb{T}$ isomorphic to $\mumu_n^d$, where $n$ is a positive integer (note that $\mathbb{T}$ contains such a subgroup for every $n$).
Let $X$ be a projective toric $\mathbb{T}$-variety of dimension~$d$.

\begin{lemma}
\label{lemma:toric-degree}\label{corollary:toric-degree}
Let $C$ be a $\Gamma_n$-invariant $\Gamma_n$-irreducible curve in $X$, and let $H$ be a very ample divisor on $X$.
If $n>H\cdot C$, then $C$ is $\mathbb{T}$-invariant.
\end{lemma}

\begin{proof}
Suppose that $C$ is not $\mathbb{T}$-invariant. By replacing $X$ by a $\mathbb{T}$-invariant toric closed subvariety if necessary,
we can assume that the curve $C$ is not contained in any proper $\mathbb{T}$-invariant subvariety of $X$ so that $\Gamma_n$ acts faithfully on $C$.
The curve $C$ is $\mathbb{T}$-invariant if and only if each of its irreducible components is $\mathbb{T}$-invariant.

Let $k$ be the number of irreducible components of the curve $C$, let $Z$ be an irreducible component of $C$ and let $\Gamma_Z$ be the stabilizer of the curve $Z$ in the~group~$\Gamma_n$. Then $\Gamma_Z$ is an index $k$ subgroup of $\mumu_n^d$, equal to the product of $d$ cyclic subgroups $\mumu_{m_i}$ for some positive integers $m_i$ which divide $n$, say $n=m_ik_i$, $i=1,\ldots, d$. Let $m=\gcd\{m_i\}_{i=1,\ldots ,d}$ and write $m_i=mr_i$ where $r_i\geqslant 1$. Then $\Gamma_Z$ contains a subgroup isomorphic to $\mumu_m^d$. By construction, we have $mn^{d-1}\geqslant \prod_{i=1}^{d} m_i$ so that
\[ k=\prod_{i=1}^{d} k_i=\prod_{i=1}^{d} \frac{n}{m_i}=\frac{n^d}{\prod_{i=1}^d m_i}\geqslant \frac{n^d}{mn^{d-1}}=\frac{n}{m}.\]
Thus $m\geqslant n/k$ and since by hypothesis $n>H\cdot C$, it follows that $m>H\cdot Z$.

Replacing $C$ by $Z$ and $n$ by $m$, we assume from now on that $C$ is irreducible. Let us show that $n\leqslant H\cdot C$. Let $f\colon\widetilde{C}\to C$ be the normalization of $C$. Then the action of the group $\Gamma_n$ lifts to a faithful action on $\widetilde{C}$. Let $D$ be a $\mathbb{T}$-invariant effective divisor such that $D\sim H$.
Then $C\not\subset\mathrm{Supp}(D)$ by assumption on $C$ not being $\mathbb{T}$-invariant. Let $\Sigma=C\cap\mathrm{Supp}(D)$,
and let $\widetilde{\Sigma}$ be its preimage in $\widetilde{C}$.
Then
$$
\big|\widetilde{\Sigma}\big|\leqslant\mathrm{deg}\Big(f^*(D\big\vert_{C})\Big)=\mathrm{deg}\Big(f^*(H\big\vert_{C})\Big)=H\cdot C.
$$
Let $P$ be a point in $\widetilde{\Sigma}$, and let $G_P$ be its stabilizer in $\Gamma_n$.
Then $G_P$ is cyclic by Lemma~\ref{corollary:stabilizer-faithful}.
On the other hand, we have
$$
|G_P|\geqslant\frac{|\Gamma_n|}{|\widetilde{\Sigma}|}\geqslant\frac{|\Gamma_n|}{H\cdot C}=\frac{n^d}{H\cdot C}\geqslant \frac{n^2}{H\cdot C}.
$$
Therefore, if $n>H\cdot C$, then the order of the cyclic group $G_P$ is strictly larger than $n$,
which is impossible, since $G_P$ is a subgroup of the group $\Gamma_n\cong\mumu_n^d$.
\end{proof}

\section{Lattices and toric geometry}
\label{section:toric}

Let $\mathbb{T}\cong\mathbb{G}_{m}^{n}$ be an algebraic torus of dimension $n$. We identify $\mathbb{T}$
with the spectrum of the group algebra $\mathbb{C}[M]$ of its character lattice $M=\mathrm{Hom}(\mathbb{T},\mathbb{G}_m)\cong \mathbb{Z}^n$.
The action of the torus $\mathbb{T}$ on itself by translations determines an injective group homomorphism $\mathbb{T}\rightarrow\mathrm{Aut}(\mathbb{T})$
and we have split exact sequence
$$
\xymatrix{1\ar@{->}[r] & \mathbb{T}\ar@{->}[r] & \mathrm{Aut}(\mathbb{T})\ar@{->}[r] & \mathrm{GL}(M)\ar@{->}[r] & 1.}
$$
The splitting is given by mapping every $A\in\mathrm{GL}(M)\cong\mathrm{GL}_{n}(\mathbb{Z})$
to the algebraic group automorphism of $\mathbb{T}$ associated to the group algebra automorphism
$\mathbb{C}[M]\to\mathbb{C}[M]$ given~by
$$
\chi^{u}\mapsto\chi^{A(u)}.
$$
We henceforth identify $\mathrm{Aut}\big(\mathbb{T}\big)=\mathbb{T}\rtimes\mathrm{GL}(M)$
and we denote its subgroup $\mathbb{T}\times\{\mathrm{id}_{M}\}$ simply by $\mathbb{T}$.
Every algebraic subgroup $G\subset\mathrm{Aut(\mathbb{T})}$
containing $\mathbb{T}$ is then of the form $G=\mathbb{T}\rtimes\mathbb{W}$
for some finite subgroup $\mathbb{W}$ of $\mathrm{GL}(M)$.

Let $\mathbb{W}_{1}$ and $\mathbb{W}_{2}$ be finite subgroups of $\mathrm{GL}(M)$,
let $G_{1}=\mathbb{T}\rtimes\mathbb{W}_{1}$ and $G_{2}=\mathbb{T}\rtimes\mathbb{W}_{2}$ be the corresponding
algebraic subgroups of $\mathrm{Aut}(\mathbb{T})$ that contain the torus $\mathbb{T}$,
let~$m_{1}\colon G_{1}\times\mathbb{T}\rightarrow\mathbb{T}$ and
$m_{2}\colon G_{2}\times\mathbb{T}\rightarrow\mathbb{T}$ be the algebraic actions they determine.

\begin{lemma}
\label{lem:conjugacy-subgroups-T}
The following conditions are equivalent:
\begin{itemize}
\item[(a)] There exist an isomorphism $\varphi\colon G_{1}\rightarrow G_{2}$ and a $\varphi$-equivariant biregular map $\Phi\colon\mathbb{T}\to\mathbb{T}$;

\item[(b)] The groups $G_{1}$ and $G_{2}$ are conjugate in $\mathrm{Aut}(\mathbb{T})$;

\item[(c)] The groups $\mathbb{W}_{1}$ and $\mathbb{W}_{2}$ are conjugate in $\mathrm{GL}(M)$.
\end{itemize}
\end{lemma}

\begin{proof}
Assume (a). Then we have a group automorphism $c_{\Phi}\colon\mathrm{Aut}(\mathbb{T})\to\mathrm{Aut}(\mathbb{T})$ given by
$\alpha\mapsto\Phi\circ\alpha\circ\Phi^{-1}$ and the hypothesis that $\Phi$ is $\varphi$-equivariant implies that the diagram
$$
\xymatrix{
G_{1}\ar@{^{(}->}[rr]\ar@{->}[d]_{\varphi} && \mathrm{Aut}\big(\mathbb{T}\big)\ar@{->}[d]^{c_{\Phi}}\\
G_{2}\ar@{^{(}->}[rr] && \mathrm{Aut}\big(\mathbb{T}\big)}
$$
must be commutative, so that the algebraic groups $G_{1}$ and $G_{2}$ are conjugate in $\mathrm{Aut}(\mathbb{T})$.
This shows that (a) implies~(b).

Now we assume (b). Then $G_{2}=\Phi G_{1}\Phi^{-1}$
for some $\Phi=(\lambda,A)$ in $\mathrm{Aut}(\mathbb{T})$.
Then $\mathbb{W}_{2}=A\mathbb{W}_{1}A^{-1}$ in $\mathrm{GL}(M)$.This shows that (b) implies (c).

Assume (c). Then $\mathbb{W}_{2}=A\mathbb{W}_{1}A^{-1}$ for some $A\in\mathrm{GL}(M)$.
Let $\Phi=(1,A)\in\mathrm{Aut}(\mathbb{T})$, and let $\varphi\colon G_{1}\to G_{2}$ be the homomorphism defined by
$g_{1}\mapsto\Phi g_{1}\Phi^{-1}$. Then $\varphi$ is an isomorphism for which we have the same commutative diagram as above.
It follows in turn that the pair $(\varphi\colon G_{1}\rightarrow G_{2},\Phi\colon\mathbb{T}\rightarrow\mathbb{T})$
is an equivariant isomorphism, which proves that (c) implies (a).
\end{proof}

Now we fix a finite subgroup $\mathbb{W}$ in $\mathrm{GL}(M)$ and we let $G=\mathbb{T}\rtimes\mathbb{W}$ be the corresponding algebraic subgroup in $\mathrm{Aut}(\mathbb{T})$ that contains $\mathbb{T}$. The group $\mathbb{W}$ acts naturally on the vector space
$$
N_{\mathbb{Q}}=\mathrm{Hom}\big(M,\mathbb{Z}\big)\otimes\mathbb{Q}.
$$
By \cite[Chapter~2]{Cox}, the choice of a $\mathbb{W}$-invariant convex lattice polytope in
$N_{\mathbb{Q}}$ determines a projective toric variety $X$ with an open $\mathbb{T}$-orbit $\mathbb{T}_{X}\cong\mathbb{T}$
such that the $G$-action on the torus $\mathbb{T}_X$ extends to a faithful regular $G$-action $m_{X}\colon G\times X\rightarrow X$.
Thus, we can identify $G$ with its image in the group $\mathrm{Aut}(X)$ by the injective
group homomorphism $\rho_{X}\colon G\rightarrow\mathrm{Aut}(X)$ given by $g\mapsto m_{X}(g,\cdot)$.
Then $\mathrm{Aut}(X)$ is an affine algebraic group having $\mathbb{T}$ as a maximal torus \cite{Demazure}.

Let $G_X$ be the normalizer of the torus $\mathbb{T}$ in the group $\mathrm{Aut}(X)$.
Then $G_X$ is an algebraic group that contains $G$. Moreover, the torus $\mathbb{T}_{X}$ is $G_{X}$-invariant,
and the induced effective action of the group $G_{X}$ on the torus $\mathbb{T}_{X}$ corresponds
to an injective group homomorphism
$$
G_{X}\rightarrow\mathrm{Aut}\big(\mathbb{T}_{X}\big)\cong\mathrm{Aut}\big(\mathbb{T}\big),
$$
whose image is equal to $\mathbb{T}\rtimes\mathbb{W}_{X}$
for a finite subgroup \mbox{$\mathbb{W}_{X}\subset\mathrm{GL}(M)$} that contains $\mathbb{W}$.
Thus, we have the following commutative diagram of exact sequences:
$$
\xymatrix{1\ar@{->}[r] & \mathbb{T}\ar@{->}[r]\ar@{=}[d] & G=\mathbb{T}\rtimes\mathbb{W}\ar@{->}[r]\ar@{_{(}->}[d] & \mathbb{W}\ar@{->}[r]\ar@{_{(}->}[d] & 1\\
1\ar@{->}[r] & \mathbb{T}\ar@{->}[r] & G_{X}\ar@{->}[r]^{\nu_X} & \mathbb{W}_{X}\ar@{->}[r] & 1.}
$$
The group $\mathbb{W}_{X}$ is usually called the \emph{Weyl group} of the toric variety $X$.

\begin{corollary}
\label{corollary:maximal}
If $\mathbb{W}$ is a maximal finite subgroup of $\mathrm{GL}(M)$ then \mbox{$G_{X}\cong\mathbb{T}\rtimes\mathbb{W}$.}
\end{corollary}

As a consequence of Lemma \ref{lem:conjugacy-subgroups-T}, we obtain the following two assertions:

\begin{corollary}
\label{corollary:correspondence}
There exists a functorial one-to-one correspondence between finite subgroups $\mathbb{W}\subset\mathrm{GL}(M)$ up to conjugacy
and projective toric $\mathbb{T}$-varieties $X$ whose Weyl groups contain a subgroup isomorphic to $\mathbb{W}$
up to $\mathbb{T}\rtimes\mathbb{W}$-equivariant birational equivalence.
\end{corollary}

\begin{corollary}
\label{corollary:Bir}
Let $\mathbb{W}$ be a finite subgroup in $\mathrm{GL}(M)$ and let $X$ be a projective toric variety whose Weyl group $\mathbb{W}_{X}$
contains $\mathbb{W}$. Then
$$
\mathrm{Bir}^{\mathbb{T}\rtimes\mathbb{W}}(X)\cong\mathbb{T}\rtimes\widehat{\mathbb{W}}
$$
where $\widehat{\mathbb{W}}$ is the normalizer of the group $\mathbb{W}$ in $\mathrm{GL}(M)$.
\end{corollary}

Given a subgroup $\mathbb{W}\subset\mathrm{GL}(M)$, we say that the lattice $M\cong\mathbb{Z}^{n}$ is $\mathbb{W}$-irreducible (or an irreducible $\mathbb{W}$-module) if $M$ does not contain any proper $\mathbb{W}$-invariant sublattice $M^\prime$ such that $M/M^\prime$ is torsion free.

\begin{corollary}
\label{corollary:Bir-maximal}
Let $\mathbb{W}$ be a maximal finite subgroup in $\mathrm{GL}(M)$ and let $X$ be a projective toric variety whose Weyl group is $\mathbb{W}$.
Suppose that $M$ is \mbox{$\mathbb{W}$-irreducible}. Then
$$
\mathrm{Bir}^{\mathbb{T}\rtimes\mathbb{W}}(X)\cong\mathbb{T}\rtimes\mathbb{W}.
$$
\end{corollary}

\begin{proof}
Since $M$ is $\mathbb{W}$-irreducible, $M\otimes\mathbb{Q}$ is an irreducible $\mathbb{Q}$-representation of the group~$\mathbb{W}$.
Applying  Maschke's theorem, we conclude that the centralizer of $\mathbb{W}$ in $\mathrm{GL}(M)$ is finite.
Since  $\mathbb{W}$ is finite, the normalizer $\widehat{\mathbb{W}}$ of the group $\mathbb{W}$ in $\mathrm{GL}(M)$ is also finite, and hence, $\widehat{\mathbb{W}}=\mathbb{W}$ because $\mathbb{W}$ is a maximal finite subgroup. The assertion then follows from Corollary~\ref{corollary:Bir}.
\end{proof}

The choice of the $n$-dimensional toric variety $X$ whose Weyl group contains $\mathbb{W}$ is not unique.
In particular, taking a $G$-equivariant toric resolution of singularities and then applying the $G$-equivariant toric Minimal Model Program,
we can assume that:
\begin{itemize}
\item The toric variety $X$ has terminal singularities,
\item Every $G$-invariant Weil divisor in $X$ is a $\mathbb{Q}$-Cartier divisor,
\item There exists a $G$-Mori fibre space structure $\pi\colon X\to Z$ (see \cite[Definition~1.1.5]{CheltsovShramov}).
\end{itemize}
In particular, if $Z$ is a point, then $X$ is a toric Fano variety with terminal singularities,
and $X$ is $G$-minimal, i.e. the group of $G$-invariant Weil divisors is of rank $1$.

Since $\pi\colon X\to Z$ is a surjective morphism of toric varieties, it induces a surjective $G$-equivariant morphism
$\mathbb{T}_X\rightarrow\mathbb{T}_Z$ between the corresponding open orbits in $X$ and $Z$, which is a group homomorphisms when we identify these orbits with the corresponding maximal tori $\mathbb{T}$ of $\mathrm{Aut}(X)$ and $\mathbb{T}'$ of $\mathrm{Aut}(Z)$ respectively.
The kernel of this homomorphism is a $\mathbb{W}$-invariant subtorus in $\mathbb{T}$, whose character lattice is a $\mathbb{W}$-invariant sublattice of the lattice $M$. This gives

\begin{corollary}
\label{corollary:not-G-solid}
If $\mathrm{dim}(Z)\geqslant 1$ then the lattice $M$ is not $\mathbb{W}$-irreducible.
\end{corollary}

In fact, we can say more:

\begin{proposition}
\label{proposition:G-solid}
Assume that $Z$ is a point. Then the following are equivalent:
\begin{enumerate}
\item[(a)] The toric Fano variety $X$ is $G$-solid;
\item[(b)] The character lattice $M$ is $\mathbb{W}$-irreducible.
\end{enumerate}
\end{proposition}

\begin{proof}
The implication (b)$\Rightarrow$(a) follows from Corollary~\ref{corollary:not-G-solid}.
Let us prove that (a)$\Rightarrow$(b). Assume that $X$ is $G$-solid and suppose that $M$ is not $\mathbb{W}$-irreducible.
Then $M$ contains a  proper $\mathbb{W}$-invariant sublattice $M^\prime$ such that $M/M^\prime$ is a torsion free $\mathbb{W}$-module.
This implies that the torus $\mathbb{T}$ contains a proper $G$-invariant subtorus $\mathbb{T}^\prime$,
which gives an exact $G$-equivariant sequence of tori
$$
1\longrightarrow\mathbb{T}^\prime\longrightarrow\mathbb{T}\longrightarrow \mathbb{T}^{\prime\prime}\longrightarrow 1,
$$
where $\mathbb{T}^{\prime\prime}\cong\mathbb{T}\slash\mathbb{T}^{\prime}$.
This gives us a $G$-equivariant dominant rational map  $\psi\colon X\dasharrow X^{\prime\prime}$,
where $X^{\prime\prime}$ is a $G$-equivariant projective completion of the torus $\mathbb{T}^{\prime\prime}$.

Then there exists a $G$-equivariant commutative diagram
$$
\xymatrix{ &\widetilde{X}\ar@{->}[dl]_{\alpha}\ar@{->}[dr]^{\beta}&\\
X\ar@{-->}[rr]_{\psi}&&X^{\prime\prime}}
$$
such that $\alpha$ is a $G$-equivariant birational morphism, $\widetilde{X}$ is a smooth projective toric variety,
and $\beta$ is a surjective $G$-equivariant morphism.
Note that
$$
\mathrm{dim}\big(X\big)>\mathrm{dim}\big(\mathbb{T}^{\prime\prime}\big)\geqslant 1.
$$
Now, we can apply a $G$-equivariant Minimal Model Program to $\widetilde{X}$ over the variety~$X^{\prime\prime}$.
This gives a $G$-equivariant birational transformation of the variety $X$
into a $G$-Mori fibre space over a positive dimensional base, which is impossible, since $X$ is $G$-solid.
\end{proof}

Thus, if $X$ is a $G$-minimal toric Fano variety, we have a purely group theoretical criterion for its $G$-solidity.
Similarly, we can obtain a criterion for $G$-birational rigidity.

\begin{proposition}
\label{proposition:G-rigid}
Let $X$ be a $G$-minimal toric Fano variety with Weyl group $\mathbb{W}_X$. Assume that the character lattice $M$ is $\mathbb{W}_X$-irreducible.
Then the following two conditions are equivalent:
\begin{enumerate}
\item[(a)] $X$ is $G$-birationally rigid;
\item[(b)] $X$ is the only toric Fano variety with terminal singularities that is $G$-minimal.
\end{enumerate}
\end{proposition}

\begin{proof}
This follows from Proposition~\ref{proposition:G-solid} and definition of $G$-birational rigidity.
\end{proof}

Finally, using Corollary~\ref{corollary:Bir}, we can obtain a criterion for $G$-birational super-rigidity.

\begin{proposition}
\label{proposition:G-super-rigid}
Let $X$ be a $G$-minimal toric Fano variety with Weyl group $\mathbb{W}$. Assume that the character lattice $M$ is $\mathbb{W}$-irreducible.
Then $X$ is $G$-birationally super-rigid if and only if the following two conditions are satisfied:
\begin{enumerate}
\item[(a)] $X$ is the only toric Fano variety with terminal singularities that is $G$-minimal;
\item[(b)] $\mathbb{W}$ is not a proper normal subgroup of any finite subgroup in $\mathrm{GL}(M)$.
\end{enumerate}
\end{proposition}

\begin{proof}
The assertion follows from Proposition~\ref{proposition:G-rigid} and the proof of Corollary~\ref{corollary:Bir-maximal}.
\end{proof}

The condition (b) in Proposition~\ref{proposition:G-rigid} is combinatorial.
A priori, it can be checked using computer, since there are finitely many toric Fano varieties with terminal singularities \cite{borisovs}.
For example, there are $634$ toric Fano threefolds with terminal singularities~\cite{Kasprzyk2006}.

\subsection{Toric terminal Fano threefolds.}
\label{subsection:toric-Fano-threefolds}
Now let us assume that $\mathbb{T}$ is three-dimensional and that $X$ is a $G$-minimal toric Fano threefold with terminal singularities.
All such threefolds are described in \cite{Sarikyan}. They are listed in the following table:

\begin{center}
\renewcommand{\arraystretch}{1.5}
\begin{tabular}{|c|c|c|}
  \hline
 Toric Fano threefold & Weyl group & Number in \cite{grdb} \\
  \hline
  \hline
Divisor $Y_{24}$ of type $(1,1,1,1)$ in $\mathbb{P}^1\times\mathbb{P}^1\times\mathbb{P}^1\times\mathbb{P}^1$ & $\mathfrak{S}_4\times\mumu_2$ & \textnumero{625}\\
  \hline
$V_6=\mathbb{P}^1\times\mathbb{P}^1\times\mathbb{P}^1$  & $\mathfrak{S}_3\ltimes\mumu_2^3\cong\mathfrak{S}_4\times\mumu_2$ & \textnumero{62} \\
  \hline
Toric Fano--Enriques threefold $X_{24}$   & $\mathfrak{S}_4\times\mumu_2$ &  \textnumero{47}\\
  \hline
Toric complete intersection $V_4\subset\mathbb{P}^5$ of two quadrics  & $\mathfrak{S}_4\times\mumu_2$ & \textnumero{297}\\
  \hline
Three-dimensional projective space  $\mathbb{P}^3$  & $\mathfrak{S}_4$ & \textnumero{4} \\
  \hline
Quadric cone in $\mathbb{P}^4$ with one singular point & $\mathrm{D}_{8}$ & \textnumero{32} \\
  \hline
Terminal toric Fano threefold $X$ with $-K_X^3=\frac{81}{2}$  & $\mathfrak{S}_3$ & \textnumero{92} \\
  \hline
Weighted projective space $\mathbb{P}(1,1,1,2)$  & $\mathfrak{S}_3$  & \textnumero{7}\\
  \hline
Quotient of the space $\mathbb{P}^3$ by $\mumu_5$-action fixing $5$ points& $\mumu_2^2$ & \textnumero{1} \\
  \hline
Weighted projective space $\mathbb{P}(1,1,2,3)$  & $\mumu_2$ & \textnumero{8} \\
  \hline
\end{tabular}
\end{center}

\medskip

\begin{proposition} Let $X$ be one of the toric Fano threefolds in the above table, let $G$ be a subgroup of $G_X$ containing $\mathbb{T}$ and let $\mathbb{W}$ be the image of $G$ by the quotient morphism $G_X\rightarrow \mathbb{W}_X=G_X/\mathbb{T}$. Then the following hold:
\begin{enumerate}
\item None of the last five threefolds in the table above is $G$-solid.
\item If $X$ is $G$-minimal then $X$ is $G$-solid if and only if it is one of the threefolds $Y_{24}$, $V_6$,  $X_{24}$, $V_{4}$ and $\mathbb{P}^3$ and $\mathbb{W}$ contains a subgroup isomorphic to $\mathfrak{A}_4$.
\end{enumerate}
\end{proposition}

\begin{proof} This follows from  Proposition~\ref{proposition:G-solid} and the classification of finite subgroups in $\mathrm{GL}_3(\mathbb{Z})$~\cite{Tahara}.
\end{proof}

Of course, it is also possible to verify these properties explicitly for each case in the above proposition. For instance:

\begin{example}
\label{example:92}
Let $X$ be the terminal toric Fano threefold~\textnumero{92}. Then $\mathbb{W}_X\cong\mathfrak{S}_3$ and there exists a $G_X$-Sarkisov link
$$
\xymatrix{
&X_{40}\ar@{->}[dl]_{\alpha}\ar@{->}[dr]^{\beta}&\\
\mathbb{P}^3 && X}
$$
where $\alpha$ is the blow-up of three coplanar $\mathbb{T}$-invariant lines,
$X_{40}$ is a Fano threefold with three ordinary double points such that \mbox{$-K_{X_{40}}^3=40$},
and $\beta$ is the contraction of the proper transform of the~unique $\mathbb{T}$-invariant plane containing the lines blown-up
to the~unique singular point of type $\frac{1}{2}(1,1,1)$ of the threefold $X$.
Since $\mathbb{P}^3$ is not $\mathbb{T}\rtimes \mathfrak{S}_3$-solid, we conclude that $X$ is not $G_X$-solid.
\end{example}

\begin{example}
\label{example:V-6-dP-6}
Let $X=V_6$ and let $\mathbb{W}$ be the unique subgroup of $\mathbb{W}_X$ isomorphic $\mumu_3$.
Then $X$ is $G$-minimal. Moreover, it contains two $G$-fixed points such that there exists the following $G$-Sarkisov link:
$$
\xymatrix{
\overline{V}_6\ar@{->}[d]_{\alpha}\ar@{-->}[rr]^{\iota}&&\widehat{V}_6\ar@{->}[d]^{\beta}\\%
V_6&&S_6}
$$
where $\alpha$ is the blow-up of these two points,
$\iota$ is a composition of Atiyah flops of the proper transforms of all $\mathbb{T}$-invariant curves that pass through one of the points blown-up by $\alpha$,
and~$\beta$ is a $\mathbb{P}^1$-bundle over a del Pezzo surface of degree $6$.
\end{example}

\begin{example}
\label{example:X24-non-G-solid}
Let $X=X_{24}$ and let $\mathbb{W}$ be the unique subgroup of $\mathbb{W}_X$ isomorphic to $\mumu_3$.
Then $X$ is $G$-minimal. Moreover, it contains two $G$-fixed singular points such that there exists a $G$-Sarkisov link
$$
\xymatrix{
\overline{X}_{24}\ar@{->}[d]_{\alpha}\ar@{-->}[rr]^{\iota}&&\widehat{X}_{24}\ar@{->}[d]^{\beta}\\%
X_{24}&&S_6}
$$
where $\alpha$ is Kawamata blow-up of these two singular points,
 $\iota$ is a composition of Francia antiflips of the proper transforms of all $\mathbb{T}$-invariant curves
that contain one of these points, and $\beta$ is a $\mathbb{P}^1$-bundle over a del Pezzo surface of degree $6$.
The threefolds $X_{24}$, $\overline{X}_{24}$, $\widehat{X}_{24}$ are
quotients by involutions of the threefolds $V_{6}$, $\overline{V}_{6}$, $\widehat{V}_{6}$ from Example~\ref{example:V-6-dP-6}.
\end{example}

Moreover, if $X$ is one of the threefolds $Y_{24}$, $V_6$,  $X_{24}$, $V_{4}$, $\mathbb{P}^3$
and $\mathbb{W}$ contains a subgroup isomorphic to $\mathfrak{A}_4$, then $X$ is $G$-minimal except in the following two cases:
\begin{enumerate}
\item $X=V_4$, $\mathbb{W}\cong\mathfrak{S}_4$ and $G$ acts intransitively on the set of $\mathbb{T}$-invariant surfaces,
\item $X=V_4$ and $\mathbb{W}\cong\mathfrak{A}_4$.
\end{enumerate}
We show this in Corollaries~\ref{corollary:Y24-G-Fano} and \ref{corollary:X24-G-Fano} and Lemma~\ref{lemma:V4-G-Fano} below.
Summing up, we get

\begin{corollary}
\label{corollary:main-infinite}
The assertion of Theorem~\ref{theorem:main} holds in the case when $G$ is infinite.
\end{corollary}

If $\mathbb{W}$ contains a subgroup isomorphic to $\mathfrak{A}_4$, then $\mathbb{W}$ is conjugate to one of $15$ finite subgroups in $\mathrm{GL}(M)$ that are described in \cite{Tahara}. Using \cite{Kunyavskii} and notation from \cite{Tahara}, we can present these
$15$ subgroups and the corresponding $G$-minimal toric Fano threefolds with terminal singularities
in the~following table.

\begin{center}
\renewcommand{\arraystretch}{1.5}
\hspace*{-1cm}
\begin{tabular}{|c||c|c|c|c|}
  \hline
  & $\mathfrak{S}_4\times\mumu_2$ & $\quad$ $\mathfrak{S}_4$ $\quad$ & $\mathfrak{A}_4\times\mumu_2$ & $\quad$ $\mathfrak{A}_4$ $\quad$\\
  \hline
  \hline
$\mathbb{P}^1\times\mathbb{P}^1\times\mathbb{P}^1$ & $W_1$ & $W_6$ or $W_7$ & $W_1$& $W_9$ \\
  \hline
$V_4$ & $W_3$& $W_{10}$  & $W_3$& \\
  \hline
$X_{24}$ & $W_3$ & $W_{10}$ or $W_{11}$ & $W_3$ &$W_{11}$ \\
  \hline
$Y_{24}$ & $W_{2}$&  $W_8$ or $W_9$ & $W_2$& $W_{10}$ \\
  \hline
$\mathbb{P}^3$ & & $W_{11}$ & &$W_{11}$ \\
  \hline
\end{tabular}
\end{center}

\medskip

Now using Propositions~\ref{proposition:G-rigid} and \ref{proposition:G-super-rigid}, we obtain

\begin{corollary}
\label{corollary:super-rigid-infinite}
The assertion of Corollary~\ref{corollary:main} holds in the case when $G$ is infinite.
\end{corollary}

In the rest of this paper, we will give another proof of Theorem~\ref{theorem:main}(3) in the~case when the group $G$
is infinite that is independent on the classification of toric Fano threefolds with terminal singularities,
and which also applies to the case of finite groups~as~well. We also believe that this approach can be used in higher-dimensions.

\section{Toric Fano threefolds and lattices of rank three}
\label{section:Fano-threefolds}

Among the $73$ conjugacy classes of finite subgroups $\mathbb{W}$ of $\mathrm{GL}_{3}(\mathbb{Z})$ classified in \cite{Tahara},
there are $4$ maximal ones, and only three of them give rise to an irreducible action on~$\mathbb{Z}^{3}$.
In~each of these three cases, one has $\mathbb{W}\cong\mathfrak{S}_{4}\times\mumu_{2}$.
Let us describe these three conjugacy classes in terms of the actions of the group $\mathfrak{S}_{4}\times\mumu_{2}$ on certain lattices.

Let $L=\mathbb{Z}^{4}$ endowed with the faithful transitive $\mathfrak{S}_{4}$-action
given by permutations of the basis vectors $h_{1}$, $h_{2}$, $h_{3}$ and $h_{4}$.
Let $\sigma$ be the involution of the lattice $L$ such that $h_{i}\mapsto-h_{i}$
for each $i\in\{1,2,3,4\}$. Then $\sigma$ commutes with the $\mathfrak{S}_{4}$-action.
This defines a faithful action of the group $\mathfrak{S}_{4}\times\mumu_{2}$ on the lattice $L$,
which leaves invariant the sublattice spanned by the~element $h_{1}+h_{2}+h_{3}+h_{4}$.
Let
$$
M_{1}=L/\langle h_{1}+h_{2}+h_{3}+h_{4}\rangle.
$$
Then the $\mathfrak{S}_{4}\times\mumu_{2}$-action on $L$ induces an action of $\mathfrak{S}_{4}\times\mumu_{2}$ on the quotient lattice~$M_1$.
Let $e_{1}$, $e_{2}$ and $e_{3}$ be the basis of $M_1$ given by the classes of $h_{1}$, $h_{2}$ and $h_3$, respectively.
In~this basis, we have $\sigma(e_{i})=-e_{i}$ for every $i\in\{1,2,3\}$,
and for every $g\in\mathfrak{S}_{4}$, we have
$$
g\big(e_{i}\big)=\begin{cases}
e_{g(i)} & \textrm{if }g(i)\neq4\\
-e_{1}-e_{2}-e_{3} & \textrm{otherwise. }
\end{cases}
$$
We denote by $\mathbb{W}_{1}\cong\mathfrak{S}_{4}\times\mumu_{2}$
the corresponding subgroup of $\mathrm{GL}_{3}(\mathbb{Z})$.

Let $M_3$ be the dual lattice to $M_1$,
and let $e_{1}^{\vee}$, $e_{2}^{\vee}$ and $e_{3}^{\vee}$ be the basis of $M_3$ that is dual to the previously fixed basis of $M_1$.
Then $\sigma(e_{i}^{\vee})=-e_{i}^{\vee}$ for each $i\in\{1,2,3\}$.
For every $g\in\mathfrak{S}_{4}$, we have
$$
g\big(e_{i}^{\vee}\big)=\sum_{j=1}^{3}e_{i}^{\vee}\big(g^{-1}(e_{j})\big)e_{j}^{\vee}.
$$
We denote by $\mathbb{W}_{3}\cong\mathfrak{S}_{4}\times\mumu_{2}$ the corresponding subgroup of $\mathrm{GL}_{3}(\mathbb{Z})$.

Finally, let $M_2$ be the lattice $\mathbb{Z}^{3}$,
let $\mathfrak{S}_{3}$ be the subgroup in $\mathrm{GL}_3(\mathbb{Z})$ consisting of six permutation matrices, let
$$
\tau_1=\begin{pmatrix}
-1 & 0 & 0\\
0 & 1 & 0\\
0 & 0 & 1
\end{pmatrix},
\tau_2=\begin{pmatrix}
1 & 0 & 0\\
0 & -1 & 0\\
0 & 0 & 1
\end{pmatrix},
\tau_3=\begin{pmatrix}
1 & 0 & 0\\
0 & 1 & 0\\
0 & 0 & -1
\end{pmatrix},
$$
and let $\mathbb{W}_{2}\cong \mathfrak{S}_{3}\ltimes\mumu_{2}^{3}$ be the subgroup in $\mathrm{GL}_{3}(\mathbb{Z})$ that is generated by $\mathfrak{S}_3$ and involutions $\tau_1$, $\tau_2$ and $\tau_3$. Note that the subgroup generated by $\mathfrak{S}_{3}$, $\tau_{1}\tau_{2}$ and $\tau_{1}\tau_{3}$ is isomorphic to~$\mathfrak{S}_{4}$, and $\tau_1\tau_2\tau_3$ generates the center of the subgroup $\mathbb{W}_{2}$.Thus, $\mathbb{W}_{2}$ is isomorphic to the group $\mathfrak{S}_{4}\times\mumu_{2}$.

\begin{proposition}[{\cite{Tahara}}]
\label{proposition:maximal-irreducible-groups-GL-3-Z}
Let $\mathbb{W}$ be a maximal finite subgroup of the group $\mathrm{GL}_{3}(\mathbb{Z})$ such that $\mathbb{Z}^{3}$ is $\mathbb{W}$-irreducible.
Then $\mathbb{W}$ is conjugate to one of the subgroups $\mathbb{W}_{1}$, $\mathbb{W}_{2}$ or $\mathbb{W}_{3}$.
Moreover, the subgroups $\mathbb{W}_{1}$, $\mathbb{W}_{2}$, $\mathbb{W}_{3}$ are pairwise non-conjugate~in~$\mathrm{GL}_{3}(\mathbb{Z})$.
\end{proposition}

\begin{notation}
\label{notation:subgroups}
The center of each of the three finite subgroups $\mathbb{W}_i$ in $\mathrm{GL}_3(\mathbb{Z})$, $i=1,2,3$ is isomorphic to $\mumu_2$, generated by the involution
$$
\sigma=\begin{pmatrix}
-1 & 0 & 0\\
0 & -1 & 0\\
0 & 0 & -1
\end{pmatrix}.
$$
For every $i=1,2,3$, the image of the subgroup $\mathfrak{A}_4\times \{1\}$ of $\mathfrak{S}_4\times \mumu_2$ by the isomorphism $\mathrm{GL}(M_i) \cong \mathrm{GL}_3(\mathbb{Z})$ given by our choice of bases is the unique subgroup of $\mathbb{W}_i$ isomorphic to $\mathfrak{A}_4$. We denote this subgroup by $\mathbb{W}_i^{\mathfrak{A}}$. In the notation of \cite[Proposition 7]{Tahara}, these groups correspond respectively to the subgroups
$$
\begin{array}{rcl}
\mathbb{W}_1^{\mathfrak{A}}=W_{10} & = & \left\{ \begin{pmatrix} 0 & 1 & 0\\ 0 & 0 & 1\\ 1 & 0 & 0 \end{pmatrix},\; \begin{pmatrix} 0 & -1 & 1\\ 0 & -1 & 0\\ 1 & -1 & 0 \end{pmatrix}\right\} \\
\mathbb{W}_2^{\mathfrak{A}}=W_{9} & = &\left\{ \begin{pmatrix} 0 & 1 & 0\\ 0 & 0 & 1\\ 1 & 0 & 0 \end{pmatrix},\; \begin{pmatrix} -1 & 0 & 0\\ 0 & 1 & 0\\ 0 & 0 & -1 \end{pmatrix}\right\} \\
\mathbb{W}_3^{\mathfrak{A}}=W_{11} &=& \left\{ \begin{pmatrix} 0 & 1 & 0\\ 0 & 0 & 1\\ 1 & 0 & 0 \end{pmatrix},\; \begin{pmatrix} -1 & -1 & -1\\ 0 & 0 & 1\\ 0 & 1 & 0 \end{pmatrix}\right\}\\
\end{array}
$$
of $\mathrm{SL}_3(\mathbb{Z})\subset \mathrm{GL}_3(\mathbb{Z})$.

On the other hand, each of the subgroups $\mathbb{W}_i$ contains two different subgroups isomorphic to $\mathfrak{S}_4$ (see \cite[Proposition 9]{Tahara}):

1) One is the image of the subgroup $\mathfrak{S}_4 \times \{1\}$ of  $\mathfrak{S}_4 \times \mumu_2$ by the isomorphism $\mathrm{GL}(M_i) \cong \mathrm{GL}_3(\mathbb{Z})$ given by our choice of bases. We denote it by $\overline{\mathbb{W}}_i^{\mathfrak{S}}$. It is easily seen that this subgroup is not contained in $\mathrm{SL}_3(\mathbb{Z})$.

2) The second one is the intersection  $\mathbb{W}_i\cap\mathrm{SL}_3(\mathbb{Z})$. It is generated by the images under the isomorphism  $\mathrm{GL}(M_i) \cong \mathrm{GL}_3(\mathbb{Z})$ of the transpositions in the subgroup $\mathfrak{S}_4\times \{1\}$ multiplied by the element $\sigma \in \mathrm{GL}_3(\mathbb{Z})$. We denote it by $\mathbb{W}_i^{\mathfrak{S}}$.
\end{notation}

The lattice $M_3$ can be seen as the root lattice of the root system $\mathrm{A}_{3}$ endowed with the natural action of the group $\mathrm{Aut}(A_{3})\cong\mathfrak{S}_{4}\times\mumu_{2}$. Similarly, one can show that $M_1$ is the weight lattice of this root lattice, so that there is an inclusion $M_3\hookrightarrow M_1$ as a~sublattice of index~$4$. With our choice of bases, it is given by the matrix
$$
\left(\begin{array}{ccc}
2 & 1 & 1\\
1 & 2 & 1\\
1 & 1 & 2
\end{array}\right).
$$
Likewise, the lattice $M_2$ is the root lattice of the root system $B_{3}$ endowed
with the natural action of the group $\mathrm{Aut}(B_{3})\cong\mathfrak{S}_{3}\ltimes\mumu_{2}^{3}\cong\mathfrak{S}_{4}\times\mumu_{2}$.
The inclusion $M_3\hookrightarrow M_1$ factors
as the composition of two inclusions $M_3\hookrightarrow M_2$
and $M_2\hookrightarrow M_1$ as sublattices of index two.
With our choice of bases, they are given by the matrices
$$
\begin{pmatrix}
1 & 1 & 0\\
1 & 0 & 1\\
0 & 1 & 1
\end{pmatrix}\quad\textrm{and}\quad
\begin{pmatrix}
1 & 1 & 0\\
1 & 0 & 1\\
0 & 1 & 1
\end{pmatrix},
$$
respectively. All the inclusions $M_3\hookrightarrow M_2\hookrightarrow M_1$ are $\mathfrak{S}_{4}\times\mumu_{2}$-equivariant.

Let $\mathbb{T}_{1}=\mathrm{Spec}(\mathbb{C}[M_{1}])$, $\mathbb{T}_{2}=\mathrm{Spec}(\mathbb{C}[M_{2}])$
and $\mathbb{T}_{3}=\mathrm{Spec}(\mathbb{C}[M_{3}])$ be the three-dimensional tori that correspond to the lattices $M_{1}$, $M_{2}$ and $M_{3}$, respectively.
We write
$$
\mathbb{C}[M_1]=\mathbb{C}\big[	\tone{t}_{1}^{\pm1},\tone{t}_{2}^{\pm1},\tone{t}_{3}^{\pm1}\big]
$$
and identify (using the inclusions $M_3\hookrightarrow M_2\hookrightarrow M_1$)
the~algebras $\mathbb{C}[M_2]$ and $\mathbb{C}[M_3]$ with the subalgebras of the algebra $\mathbb{C}[M_1]$ as follows:
$$
\mathbb{C}[M_2]=\mathbb{C}[\ttwo{t}_{1}^{\pm1},\ttwo{t}_{2}^{\pm1},\ttwo{t}_{3}^{\pm1}]=\mathbb{C}\big[(\tone{t}_{1}\tone{t}_{2})^{\pm1},(\tone{t}_{1}\tone{t}_{3})^{\pm1},(\tone{t}_{2}\tone{t}_{3})^{\pm1}\big]
$$
and
$$
\begin{array}{rcl}
\mathbb{C}[M_3]=\mathbb{C}[t_{1}^{\pm1},t_{2}^{\pm1},t_{3}^{\pm1}] & = &\mathbb{C}\big[(\ttwo{t}_{1}\ttwo{t}_{2})^{\pm1},(\ttwo{t}_{1}\ttwo{t}_{3})^{\pm1},(\ttwo{t}_{2}\ttwo{t}_{3})^{\pm1}\big] \\
 &= &\mathbb{C}\big[(\tone{t}_{1}^{2}\tone{t}_{2}\tone{t}_{3})^{\pm1},(\tone{t}_{1}\tone{t}_{2}^{2}\tone{t}_{3})^{\pm1},(\tone{t}_{1}\tone{t}_{2}\tone{t}_{3}^{2})^{\pm1}\big].
\end{array}
$$
This gives us morphisms $q_{12}\colon\mathbb{T}_{1}\rightarrow\mathbb{T}_{2}$ and $q_{23}\colon\mathbb{T}_{2}\rightarrow\mathbb{T}_{3}$,
which are quotients by
the~involutions $(\tone{t}_{1},\tone{t}_{2},\tone{t}_{3})\mapsto(-\tone{t}_{1},-\tone{t}_{2},-\tone{t}_{3})$
and $(\ttwo{t}_{1},\ttwo{t}_{2},\ttwo{t}_{3})\mapsto(-\ttwo{t}_{1},-\ttwo{t}_{2},-\ttwo{t}_{3})$, respectively.

In the next three subsections, we present three toric Fano varieties with terminal singularities
that are natural equivariant compactifications of the tori $\mathbb{T}_1$, $\mathbb{T}_2$, $\mathbb{T}_3$ following the scheme described in Section~\ref{section:toric}.
Before doing this, let us first fix some notation.

\begin{notation}
\label{notations:P1-P1-P1}
Let $([u_{1}:v_{1}],\cdots,[u_{n}:v_{n}])$ be homogeneous coordinates on $(\mathbb{P}^{1})^{n}$.
We equip $(\mathbb{P}^{1})^{n}$ with its standard structure of a toric variety with
open orbit $\mathbb{T}_{(\mathbb{P}^{1})^{n}}$ given by
$$
\prod_{i=1}^{n}u_{i}v_{i}\neq 0.
$$
We view the collection of ratios $\big(\frac{u_{1}}{v_{1}},\ldots,\frac{u_{n}}{v_{n}}\big)$
as natural ``toric coordinates'' on $(\mathbb{P}^{1})^{n}$. We
use these to identify each torus-invariant irreducible closed
subvariety of $(\mathbb{P}^{1})^{n}$ with the toric coordinates of its general point. For example, for $n=4$, this yields:
\begin{itemize}
\item $(0,1,1,1)$ is the torus-invariant divisor $u_{1}=0$;

\item $(0,1,\infty,1)$ is the torus-invariant surface
given by $u_{1}=v_{3}=0$;

\item $(0,0,0,0)$ is the torus-invariant point $u_{1}=u_{2}=u_{3}=u_{4}=0$.
\end{itemize}

Finally, we denote by $\upsilon$ be the involution of $\mathbb{P}^{1}$
given by $[u:v]\mapsto[v:u]$.
\end{notation}

\subsection{Toric Fano threefold with Weyl group $\mathbb{W}_1$}
\label{section:Y-24}
The convex hull of the points
$$
(0,\pm1,\pm1), (\pm1,0,\pm1), (\pm1,\pm1,0)
$$
in $\mathrm{Hom}(M_1,\mathbb{Z})\otimes\mathbb{Q}$ is a $\mathbb{W}_{1}$-invariant convex polytope.
One can show that the~associated toric Fano threefold is the hypersurface $Y_{24}$ in $(\mathbb{P}^{1})^{4}$ that is given by
$$
u_{1}u_{2}u_{3}u_{4}-v_{1}v_{2}v_{3}v_{4}=0.
$$
The open $\mathbb{T}_{1}$-orbit is the subset $\mathbb{T}_{Y_{24}}$ that is given by
$$
u_{1}u_{2}u_{3}u_{4}v_{1}v_{2}v_{3}v_{4}\neq0.
$$
We have $\mathbb{W}_{Y_{24}}\cong\mathbb{W}_{1}$, so that we identify
$\mathbb{W}_{Y_{24}}=\mathbb{W}_{1}$, $\mathbb{T}_{Y_{24}}=\mathbb{T}_{1}$ and $G_{Y_{24}}=\mathbb{T}_{1}\rtimes\mathbb{W}_{1}$.

The $\mathbb{W}_{1}$-action on $Y_{24}$ is given by the permutations of the factors in $(\mathbb{P}^{1})^{4}$
and the~involution $\upsilon\times\upsilon\times\upsilon\times\upsilon$,
which corresponds to the element $\sigma$ of $\mathbb{W}_1$.
We also denote this involution by $\sigma_{Y_{24}}$.

The threefold $Y_{24}$ has fourteen $\mathbb{T}_{1}$-fixed points:
the six points
$$
(0,0,\infty,\infty),(0,\infty,0,\infty),(0,\infty,\infty,0),(\infty,0,\infty,0),(\infty,0,0,\infty),(\infty,\infty,0,0),
$$
which are isolated ordinary double points forming the singular locus
of $Y_{24},$ and the eight smooth points
$$
\begin{array}{cccc}
(0,\infty,\infty,\infty), & (\infty,0,\infty,\infty), & (\infty,\infty,0,\infty), & (\infty,\infty,\infty,0),\\
(\infty,0,0,0), & (0,\infty,0,0), & (0,0,\infty,0), & (0,0,0,\infty).
\end{array}
$$
Similarly, it has twenty four irreducible $\mathbb{T}_{1}$-invariant
curves
$$
\begin{array}{cccccc}
(0,0,1,\infty), & (0,1,0,\infty), & (1,0,0,\infty), & (0,0,\infty,1), & (0,1,\infty,0), & (1,0,\infty,0),\\
(0,\infty,0,1), & (0,\infty,1,0), & (1,\infty,0,0), & (\infty,0,0,1), & (\infty,0,1,0), & (\infty,1,0,0),\\
(\infty,\infty,1,0), & (\infty,1,\infty,0), & (1,\infty,\infty,0), & (\infty,\infty,0,1), & (\infty,1,0,\infty), & (1,\infty,0,\infty),\\
(\infty,0,\infty,1), & (\infty,0,1,\infty), & (1,0,\infty,\infty), & (0,\infty,\infty,1), & (0,\infty,1,\infty), & (0,1,\infty,\infty),
\end{array}
$$
and twelve irreducible $\mathbb{T}_{1}$-invariant surfaces
$$
\begin{array}{cccccc}
(0,1,1,\infty), & (1,0,1,\infty), & (1,1,0,\infty), & (0,1,\infty,1), & (1,0,\infty,1), & (1,1,\infty,0),\\
(0,\infty,1,1), & (1,\infty,0,1), & (1,\infty,1,0), & (\infty,0,1,1), & (\infty,1,0,1), & (\infty,1,1,0).
\end{array}
$$

With this description, the following lemma is straightforward to check.

\begin{lemma}
\label{lemma:Y24-G-Fano}
Let $\mathbb{W}$ be a subgroup in $\mathbb{W}_{1}$ that contains $\mathbb{W}_{1}^{\mathfrak{A}}$.
Then the following~hold:
\begin{enumerate}
\item The group $\mathbb{W}_1^{\mathfrak{A}}$ acts transitively on the set of $\mathbb{T}_{1}$-invariant surfaces and on the set of singular points of $Y_{24}$.

\item The groups $\mathbb{W}_{1}^{\mathfrak{A}}$ and $\overline{\mathbb{W}}_{1}^{\mathfrak{S}}$ act on the set of smooth $\mathbb{T}_{1}$-fixed points and on the set of  $\mathbb{T}_{1}$-invariant curves with the same orbits. The action on the set of smooth $\mathbb{T}_{1}$-fixed points has two orbits:
one consisting of the points
$$
(0,\infty,\infty,\infty),(\infty,0,\infty,\infty),(\infty,\infty,0,\infty),(\infty,\infty,\infty,0),
$$
and another one consisting of the points
$$
(\infty,0,0,0),(0,\infty,0,0),(0,0,\infty,0),(0,0,0,\infty),
$$

Similarly, the action on the set of irreducible $\mathbb{T}_{1}$-invariant curves has two orbits: one consisting of the curves
$$
\begin{array}{cccccc}
(0,0,1,\infty), & (0,1,0,\infty), & (1,0,0,\infty), & (0,0,\infty,1), & (0,1,\infty,0), & (1,0,\infty,0),\\
(0,\infty,0,1), & (0,\infty,1,0), & (1,\infty,0,0), & (\infty,0,0,1), & (\infty,0,1,0), & (\infty,1,0,0),
\end{array}
$$
and the other one consisting of the curves
$$
\begin{array}{cccccc}
(\infty,\infty,1,0), & (\infty,1,\infty,0), & (1,\infty,\infty,0), & (\infty,\infty,0,1), & (\infty,1,0,\infty), & (1,\infty,0,\infty),\\
(\infty,0,\infty,1), & (\infty,0,1,\infty), & (1,0,\infty,\infty), & (0,\infty,\infty,1), & (0,\infty,1,\infty), & (0,1,\infty,\infty).
\end{array}
$$

\item The group $\mathbb{W}_{1}^{\mathfrak{S}}$ acts transitively on the set of smooth $\mathbb{T}_{1}$-fixed points and on the set of irreducible $\mathbb{T}_{1}$-invariant curves.

\item If $\sigma_{Y_{24}}\in\mathbb{W}$, then $\mathbb{W}$ acts transitively on the set of smooth $\mathbb{T}_{1}$-fixed points and on the set of $\mathbb{T}_{1}$-invariant curves of $Y_{24}$.

\end{enumerate}
\end{lemma}

\begin{corollary}
\label{corollary:Y24-G-Fano}
Let $G$ be a subgroup of $G_{Y_{24}}$ that contains $\mathbb{W}_{1}^{\mathfrak{A}}$. Then $\mathrm{rk}(\mathrm{Cl}(Y_{24})^{G})=1$.
\end{corollary}

\subsection{Toric Fano threefold with Weyl group $\mathbb{W}_2$}
\label{subsection:P1-P1-P1}

The convex hull of the lattice points $(0,0,\pm 1)$,$(0,\pm 1,0)$,$(\pm 1,0,0)$
in $\mathrm{Hom}(M_2,\mathbb{Z})\otimes\mathbb{Q}$ is a $\mathbb{W}_{2}$-invariant convex polytope.
One can check that the associated toric Fano threefold is $V_{6}=\mathbb{P}^{1}\times\mathbb{P}^{1}\times\mathbb{P}^{1}$.
Moreover, one has $\mathbb{W}_{V_{6}}\cong\mathbb{W}_{2}$.
Therefore, we identify
$\mathbb{W}_{V_{6}}=\mathbb{W}_{2}$, $\mathbb{T}_{V_{6}}=\mathbb{T}_{2}$ and $G_{V_{6}}=\mathbb{T}_{2}\rtimes\mathbb{W}_{2}$.

The action of $\mathbb{W}_{2}\cong\mathfrak{S}_{3}\ltimes\mumu_{3}^{2}$ on the threefold $V_{6}$ is given by
the permutations of three factors, the involutions $\upsilon\times\upsilon\times\mathrm{id}_{\mathbb{P}^{1}}$
and $\upsilon\times\mathrm{id}_{\mathbb{P}^{1}}\times\upsilon$, and
the involution $\upsilon\times\upsilon\times\upsilon$,
which corresponds to the element $\sigma$ of $\mathbb{W}_2$.
We also denote this involution by $\sigma_{V_{6}}$.

\begin{lemma}
\label{lemma:V6-G-Fano}
Let $\mathbb{W}$ be a subgroup of $\mathbb{W}_2$ that contains $\mathbb{W}_{2}^{\mathfrak{A}}$.
Then the following hold:
\begin{enumerate}
\item The group $\mathbb{W}_2^{\mathfrak{A}}$ acts transitively on the set of irreducible $\mathbb{T}_{2}$-invariant surfaces and on the set of irreducible $\mathbb{T}_{2}$-invariant curves.
\item The groups $\mathbb{W}_{2}^{\mathfrak{A}}$ and $\overline{\mathbb{W}}_{2}^{\mathfrak{S}}$ act on the set of $\mathbb{T}_{2}$-fixed points with the same two orbits: one consisting of the points
$$
(0,0,0), (\infty,\infty,0), (\infty,0,\infty), (0,\infty,\infty),
$$
and another one consisting of the points
$$
(\infty,\infty,\infty), (0,0,\infty), (0,\infty,0), (\infty,0,0).
$$
\item The group $\mathbb{W}_{2}^{\mathfrak{S}}$ acts transitively on the set of $\mathbb{T}_{2}$-fixed points.
\item If $\sigma_{V_{6}}\in\mathbb{W}$, then $\mathbb{W}$ acts transitively on the set of $\mathbb{T}_{2}$-fixed points.
\end{enumerate}
\end{lemma}

\begin{corollary}
\label{corollary:V6-G-Fano}
Let $G$ be a subgroup of $G_{V_{6}}$ that contains $\mathbb{W}_{2}^{\mathfrak{A}}$. Then $\mathrm{rk}(\mathrm{Cl}(V_{6})^{G})=1$.
\end{corollary}

\begin{remark}
Recall that we have the quotient morphism $q_{12}\colon\mathbb{T}_{1}\to\mathbb{T}_{2}$,
which is given by
$$
(\tone{t}_{1},\tone{t}_{2},\tone{t}_{3})\mapsto(\tone{t}_{1}\tone{t}_{2},\tone{t}_{1}\tone{t}_{3},\tone{t}_{2}\tone{t}_{2}).
$$
By construction, this morphism is equivariant for the actions of the group $\mathfrak{S}_{4}\times\mumu_{2}$
given by the subgroups $\mathbb{W}_{1}$ and $\mathbb{W}_{2}$, respectively.
Moreover, it induces a $q_{12}$-equivariant rational map $\varphi\colon Y_{24}\dashrightarrow V_{6}$ that has generic degree $2$.
This rational map is equivariant for the~actions of the group $\mathfrak{S}_{4}\times\mumu_{2}$ on the threefolds $Y_{24}$ and $V_{6}$.
With our choice of coordinates, this rational map can be explicitly written as follows.
For $i\in\{1,2,3,4\}$, let
$$
U_{i}=\frac{u_{i}}{v_i}u_{1}u_{2}u_{3}u_{4}v_{1}v_{2}v_{3}v_{4}\quad\textrm{and}\quad V_{i}=\frac{v_{i}}{u_i}u_{1}u_{2}u_{3}u_{4}v_{1}v_{2}v_{3}v_{4}.
$$
Then the image of the rational $\Phi\colon Y_{24}\dashrightarrow\mathbb{P}^{7}$ defined by
\begin{equation}
\label{equation:Y24-V6-P7}
\Big(\big[u_1:v_1\big],\big[u_2:v_2\big],\big[u_3:v_3\big],\big[u_4:v_4\big]\Big)\mapsto\big[U_1:V_1:U_2:V_2:U_3:V_3:U_4:V_4\big]
\end{equation}
is equal to the image of $V_{6}$ by the Segre embedding $j:V_{6}\hookrightarrow\mathbb{P}^{7}$ given by
\begin{multline*}
\Big([u_{1}:v_{1}],[u_{2}:v_{2}],[u_{3}:v_{3}]\Big)\mapsto \Big[u_{1}u_{2}v_{3}:v_{1}v_{2}u_{3}:u_{1}v_{2}u_{3}:v_{1}u_{2}v_{3}:\\
 v_{1}u_{2}v_{3}:v_{1}u_{2}u_{3}:u_{1}v_{2}v_{3}:v_{1}v_{2}v_{3}:u_{1}u_{2}u_{3}\Big]
\end{multline*}
and $j\circ\varphi=\Phi$. Moreover, there exists $\mathfrak{S}_{4}\times\mumu_{2}$-equivariant commutative diagram
\begin{equation}
\label{equation:bad-link-Y-24}
\xymatrix{
\widetilde{Y}_{24}\ar@{->}[d]_{\alpha}\ar@{-->}[rrrr]^{\beta} &&&& \widehat{Y}_{24}\ar@{->}[d]^{\gamma}\\
Y_{24}\ar@{-->}[rr]^{\varphi} && V_6 &&  Y_{12}\ar@{->}[ll]_{\delta}}
\end{equation}
where $\alpha$ is the blow-up of all the singular points of $Y_{24}$,
$\beta$  is the composition of Atiyah flops of the proper transforms
of all $\mathbb{T}_{1}$-invariant curves in $Y_{24}$, $\gamma$ is the contraction
of the proper transforms of all $\mathbb{T}_{1}$-invariant surfaces in $Y_{24}$,
and $\delta$ is the double cover branched over the~union of all $\mathbb{T}_{2}$-invariant surfaces.
The~threefold $Y_{12}$ is the canonical toric Fano threefold \textnumero{525553} in \cite{grdb}.
\end{remark}

\subsection{Toric Fano threefold with Weyl group $\mathbb{W}_3$}
\label{subsection:Fano-Enriques}
The convex hull of the eight points
$$
(\pm1,\pm1,\pm1)
$$
in $\mathrm{Hom}(M_3,\mathbb{Z})\otimes\mathbb{Q}$ is a $\mathbb{W}_3$-invariant convex polytope.
One can check that the associated toric Fano threefold is the toric Fano threefold \textnumero{47} in \cite{grdb}. Following
\cite[$\S$ 6.3.2]{Bayle}, we can also view $X_{24}$ as the quotient
of the threefold $V_{6}$ by the involution $\tau_{V_{6}}\colon V_{6}\to V_{6}$
defined~by
\begin{equation}
\Big(\big[u_{1}:v_{1}\big],\big[u_{2},v_{2}\big],\big[u_{3},v_{3}\big]\Big)\mapsto\Big(\big[u_{1}:-v_{1}\big],\big[u_{2}:-v_{2}\big],\big[u_{3}:-v_{3}\big]\Big).\label{eq:tau-V6}
\end{equation}
In this presentation, the threefold  $X_{24}$ comes with a closed embedding $X_{24}\hookrightarrow\mathbb{P}^{13}$
which is induced by the rational map $V_{6}\dashrightarrow\mathbb{P}^{13}$ defined by
\begin{equation}
\begin{array}{ccc}
\big([u_{1}:v_{1}],[u_{2}:v_{2}],[u_{3}:v_{3}]\big) & \mapsto & \Big[u_{1}^{2}u_{2}^{2}u_{3}^{2}:u_{1}^{2}u_{2}^{2}v_{3}^{2}:u_{1}^{2}u_{2}v_{2}u_{3}v_{3}:u_{1}^{2}v_{2}^{2}u_{3}^{2}:\\
 &  & u_{1}^{2}v_{2}^{2}v_{3}^{2}:u_{1}v_{1}u_{2}^{2}u_{3}v_{3}:u_{1}v_{1}u_{2}v_{2}u_{3}^{2}:\\
 &  & u_{1}v_{1}u_{2}v_{2}v_{3}^{2}:u_{1}v_{1}v_{2}^{2}u_{3}v_{3}:v_{1}^{2}u_{2}^{2}u_{3}^{2}:\\
 &  & v_{1}^{2}u_{2}^{2}v_{3}^{2}:v_{1}^{2}u_{2}v_{2}u_{3}v_{3}:v_{1}^{2}v_{2}^{2}u_{3}^{2}:v_{1}^{2}v_{2}^{2}v_{3}^{2}\Big].
\end{array}\label{eq:Embed-X24-P13}
\end{equation}

The action of the torus $\mathbb{T}_{3}$ on the threefold $X_{24}$
coincides with that induced from the~action of the torus $\mathbb{T}_{2}$ on
the threefold $V_{6}$ via the quotient morphism $\pi\colon V_{6}\rightarrow X_{24}$.
Namely, the torus $\mathbb{T}_{3}$ is the quotient of the torus $\mathbb{T}_{2}$
by the involution
$$
\big(\ttwo{t}_{1},\ttwo{t}_{2},\ttwo{t}_{3}\big)\mapsto\big(-\ttwo{t}_{1},-\ttwo{t}_{2},-\ttwo{t}_{3}\big),
$$
and the quotient morphism $\pi\colon V_{6}\rightarrow X_{24}=V_{6}/\tau_{V_{6}}$
becomes equivariant with respect to the quotient morphism $q_{23}\colon\mathbb{T}_{2}\rightarrow\mathbb{T}_{3}$
when we equip the threefold $X_{24}$ with the induced structure of toric variety.

The involution $\tau_{V_{6}}$ commutes with the action of the Weyl group $\mathbb{W}_{V_{6}}\cong\mathbb{W}_{2}$,
so that we~have $\mathbb{W}_{X_{24}}\cong\mathbb{W}_{3}$.
Hence, we identify
$\mathbb{W}_{X_{24}}=\mathbb{W}_{3}$, $\mathbb{T}_{X_{24}}=\mathbb{T}_{3}$ and $G_{X_{24}}=\mathbb{T}_{3}\rtimes\mathbb{W}_{3}$.

The action of the group $\mathbb{W}_{X_{24}}$ on the threefold $X_{24}$ coincide with the
action induced from the action of the group $\mathbb{W}_{V_{6}}$ on the threefold $V_{6}$
via the quotient morphism $\pi\colon V_{6}\rightarrow X_{24}$.
We~denote by $\sigma_{X_{24}}$ the involution in $\mathbb{W}_{X_{24}}$ induced by $\sigma_{V_{6}}\in\mathbb{W}_{V_{6}}$.

The morphism $\pi\colon V_{6}\rightarrow X_{24}$ maps $\mathbb{W}_{V_{6}}$-orbits of irreducible $\mathbb{T}_{2}$-invariants closed subvarieties in $V_{6}$
to the $\mathbb{W}_{X_{24}}$-orbits of irreducible $\mathbb{T}_{3}$-invariant closed subvarieties in $X_{24}$.
Because of this, we will denote irreducible $\mathbb{T}_{3}$-invariant
closed subvarieties in $X_{24}$ by the~same symbols as the corresponding
irreducible $\mathbb{T}_{2}$-invariant subvarieties in $V_{6}$.

Observe that the Fano threefold $X_{24}$ has exactly eight $\mathbb{T}_{3}$-fixed points,
which are singular points of type $\frac{1}{2}(1,1,1)$. They are the images of the fixed points of the involution $\tau_{V_{6}}$.
One can check that the divisor class group of the threefold $X_{24}$ is isomorphic to $\mathbb{Z}^{3}\oplus\mathbb{Z}_2$.
It is generated by the images of the toric divisors in the threefold $V_{6}$.
We have
$$
-2K_{X_{24}}\sim\mathcal{O}_{\mathbb{P}^{13}}(2)|_{X_{24}}
$$
and $-K_{X_{24}}$ is not a Cartier divisor. This together with the adjunction formula imply that every smooth
hyperplane section of the threefold $X_{24}\subset\mathbb{P}^{13}$ is an Enriques surface.

As a consequence of Corollary \ref{corollary:V6-G-Fano}, we obtain.

\begin{corollary}
\label{corollary:X24-G-Fano}
Let $G$ be a subgroup in $G_{X_{24}}$ that contains $\mathbb{W}_{3}^{\mathfrak{A}}$. Then $\mathrm{rk}(\mathrm{Cl}(X_{24})^{G})=1$.
\end{corollary}

\section{Two equivariant Sarkisov links}
\label{section:two-Sarkisov-links}

In this section, we present two known toric birational
maps between $X_{24}$ and two other terminal toric Fano threefolds (see Lemmas~\ref{lem:GV4-Sarkisov-link} and \ref{lem:GP3-Sarkisov-link} below), which will play a central role in the proof of Theorem \ref{theorem:main}.

Let $X_{8}$ be the complete intersection of three quadrics in $\mathbb{P}^{6}$
with homogeneous coordinates $[y_{1}:y_{2}:y_{3}:y_{4}:y_{5}:y_{6}:y_{7}]$ given by the equations
\begin{equation}
\begin{cases}
y_{7}^{2}-y_{1}y_{6}=0\\
y_{7}^{2}-y_{2}y_{5}=0\\
y_{7}^{2}-y_{3}y_{4}=0.
\end{cases}\label{equation:complete-intersection-3-quadrics}
\end{equation}
We view $X_{8}$ as a toric variety for the torus $\mathbb{T}_{2}$,
with open orbit $\mathbb{T}_{X_{8}}$ that is given by $$y_{1}y_{2}y_{3}y_{4}y_{5}y_{6}y_{7}\ne 0.$$
Then $X_{8}$ has six $\mathbb{T}_{2}$-fixed points, which
are its singular points, it has twelve \mbox{$\mathbb{T}_{2}$-invariants}
irreducible curves, which are lines in $\mathbb{P}^{6}$,
and it has eight $\mathbb{T}_{2}$-invariants irreducible surfaces,
which are planes in $\mathbb{P}^{6}$.
The rational map $\mathbb{P}^{6}\dashrightarrow V_{6}$ given by
\begin{equation}
\big[y_{1}:y_{2}:y_{3}:y_{4}:y_{5}:y_{6}:y_{7}\big]\mapsto\Big(\big[y_{1}:y_{7}\big],\big[y_{2}:y_{7}\big],\big[y_{3}:y_{7}\big]\Big)\label{equation:Bir-X8-V6}
\end{equation}
induces a $\mathbb{T}_{2}$-equivariant birational map $\Phi\colon X_{8}\dashrightarrow V_{6}$
whose inverse is given by
\begin{equation}
\begin{array}{ccc}
\big([u_{1}:v_{1}],[u_{2}:v_{2}],[u_{3}:v_{3}]\big) & \mapsto & \Big[u_{1}^{2}W_{2,3}:u_{2}^{2}W_{1,3}:u_{3}^{2}W_{1,2}:v_{3}^{2}W_{1,2}:\\
 &  & v_{2}^{2}W_{1,3}:v_{1}^{2}W_{2,3}:u_{1}u_{2}u_{3}v_{1}v_{2}v_{3}\Big],
\end{array}\label{equation:Map-X8}
\end{equation}
where $W_{i,j}=u_{i}v_{i}u_{j}v_{j}$ for every $i$ and $j$ in $\{1,2,3\}$.
Moreover, there exists a commutative diagram
\begin{equation}
\xymatrix{\widetilde{V}_{6}\ar@{->}[d]^{\beta} \ar@{-->}[rr] &  &\widetilde{X}_{8}\ar@{->}[d]_{\alpha} \\
V_{6} \ar@{-->}[rr]^{\Phi^{-1}} &  & X_{8}
}
\label{eq:bad-GV6-link}
\end{equation}
where $\beta$ is the blow-up of all eight $\mathbb{T}_{2}$-fixed points,
the top dashed arrow consists of flips in the proper transforms of
the twelve $\mathbb{T}_{2}$-invariant lines in $X_{8}$,
and $\alpha$ is the~crepant contraction of the~proper transforms of the six $\mathbb{T}_{2}$-invariant surfaces in $V_{6}$.

The action of $\mathbb{W}_{V_{6}}$ on $V_{6}$ given in Subsection \ref{subsection:P1-P1-P1},
and the formulas \eqref{equation:Bir-X8-V6} and \eqref{equation:Map-X8} imply that
$$
\Phi^{-1}\mathbb{W}_{V_{6}}\Phi=\mathbb{W}_{X_{8}}\cong\mathbb{W}_{2},
$$
so that $\Phi$ is a birational $\mathbb{T}_{2}\rtimes\mathbb{W}_{2}$-map.
The diagram \eqref{eq:bad-GV6-link} is a so-called \emph{bad} $\mathbb{T}_{2}\rtimes\mathbb{W}_{2}$-Sarkisov link.

The composition $\Phi^{-1}\circ\tau_{V_{6}}\circ\Phi$ (see (\ref{eq:tau-V6})
for the definition of $\tau_{V_{6}})$ is the biregular involution
$\tau_{X_{8}}$ of the threefold $X_{8}$ defined by
$$
[y_{1}:y_{2}:y_{3}:y_{4}:y_{5}:y_{6}:y_{7}]\mapsto[y_{1}:y_{2}:y_{3}:y_{4}:y_{5}:y_{6}:-y_{7}].
$$
The projection $\mathbb{P}^{6}\dashrightarrow\mathbb{P}^{5}$ from the point $[0:0:0:0:0:0:1]$ induces an isomorphism between the quotient
$X_{8}/\tau_{X_{8}}$ and the complete intersection $V_{4}\subset\mathbb{P}^{5}$ given by
\begin{equation}
\label{equation:complete-intersection}
\begin{cases}
y_{1}y_{6}-y_{2}y_{5}=0\\
y_{1}y_{6}-y_{3}y_{4}=0.
\end{cases}
\end{equation}
We view $V_{4}$ as a toric variety for the torus $\mathbb{T}_{3}$ (see Section~\ref{section:Fano-threefolds}),
with open orbit $\mathbb{T}_{V_{4}}$ given by $$y_{1}y_{2}y_{3}y_{4}y_{5}y_{6}\ne 0,$$
so that the quotient morphism $\pi\colon X_{8}\to V_{4}$ is equivariant with respect
to the quotient morphism $q_{23}\colon\mathbb{T}_{2}\to\mathbb{T}_{3}$.

The threefold $V_{4}$ has six $\mathbb{T}_{3}$-fixed points, twelve irreducible $\mathbb{T}_{3}$-invariant curves,
which are lines in $\mathbb{P}^{5}$ and eight irreducible $\mathbb{T}_{3}$-invariant surfaces which are planes in $\mathbb{P}^{5}$.
These $\mathbb{T}_{3}$-invariant irreducible subvarieties are the images of the $\mathbb{T}_2$-invariant irreducible subvarieties of $X_8$ by the quotient morphism
$\pi\colon X_{8}\to V_{4}$.

Since $\tau_{V_{8}}$ commutes with the action of $\mathbb{T}_{2}\rtimes\mathbb{W}_{2}$ on the threefold $X_{8}$,
we obtain an induced regular action of $\mathfrak{S}_{4}\times\mumu_{2}$ on the threefold $V_{4}$.
Moreover, one has $\mathbb{W}_{V_{4}}\cong\mathbb{W}_{3}$, and the threefold $V_{4}$ endowed with the action of $G_{V_{4}}=\mathbb{T}_{3}\rtimes\mathbb{W}_{V_{4}}$
is another projective terminal toric Fano model for the subgroup $\mathbb{W}_{3}$ of $\mathrm{GL}_{3}(\mathbb{Z})$.
As usual, we identify $\mathbb{W}_{V_4}=\mathbb{W}_{3}$, we let $\sigma_{V_4}$ to be the involution in $\mathbb{W}_{V_4}$ defined by
$$
[y_{1}:y_{2}:y_{3}:y_{4}:y_{5}:y_{6}]\mapsto[y_{6}:y_{5}:y_{4}:y_{3}:y_{2}:y_{1}],
$$
and we let $\nu\colon G_{V_{4}}\to\mathbb{W}_{3}$ be the natural homomorphism.

\begin{lemma}[{\cite[Theorem~10]{Avilov}}]
\label{lemma:V4-G-Fano}
Let $G$ be a subgroup of $G_{V_{4}}$ such that $\nu(G)$ contains $\mathbb{W}_{3}^{\mathfrak{A}}$.
Then $\mathrm{rk}(\mathrm{Cl}(V_{4})^{G})=1$ if and only if $\sigma_{V_{4}}\in\nu(G)$ or $\nu(G)=\mathbb{W}_3^{\mathfrak{S}}$.
\end{lemma}
\begin{proof}
By construction, the eight irreducible $\mathbb{T}_{3}$-invariant surfaces in $V_4$ are the images by $\pi\colon X_{8}\to V_{4}$ of the eight irreducible $\mathbb{T}_{2}$-invariant surfaces in $X_8$. By the $\mathbb{T}_2\rtimes \mathbb{W}_2$-equivariant commutative diagram \eqref{eq:bad-GV6-link}, the latter are the images by $\alpha:\tilde{X}_8\rightarrow X_8$ of the proper transforms of  exceptional divisors of the blow-up $\beta:\tilde{V}_8\rightarrow V_6$ of the eight $\mathbb{T}_{2}$-fixed points of $V_6$. Since the action of $\mathbb{W}_3$ on the character lattice of $\mathbb{T}_3$ is induced from that of $\mathbb{W}_2$ on the character lattice of $\mathbb{T}_2$, it follows that $\nu(G)$ acts transitively on the eight irreducible $\mathbb{T}_{3}$-invariant surfaces in $V_4$ if and only this group acts transitively on the eight $\mathbb{T}_{2}$-fixed points of $V_6$. The assertion then follows from Lemma \ref{lemma:V6-G-Fano}.
\end{proof}

\bigskip

The birational $\mathbb{T}_{2}\rtimes\mathbb{W}_{2}$-map $\Phi$ in \eqref{eq:bad-GV6-link} induces a birational $\mathbb{T}_{3}\rtimes\mathbb{W}_{3}$-map
$\varphi:V_{4}\dashrightarrow X_{24}$ defined by
\begin{equation}
\label{eq:Bir-V4-X24}
\begin{array}{ccc}
\big[y_{1}:y_{2}:y_{3}:y_{4}:y_{5}:y_{6}\big] & \mapsto & \big[y_{1}y_{2}y_{3}:y_{1}y_{2}y_{4}:y_{1}y_{3}y_{4}:y_{1}y_{3}y_{5}:y_{1}y_{4}y_{5}:\\
 &  & y_{1}y_{2}y_{6}:y_{1}y_{3}y_{6}:y_{1}y_{4}y_{6}:y_{1}y_{5}y_{6}:\\
 &  & y_{2}y_{3}y_{6}:y_{2}y_{4}y_{6}:y_{2}y_{5}y_{6}:y_{3}y_{5}y_{6}:y_{4}y_{5}y_{6}\big],
\end{array}
\end{equation}
and we eventually obtain the following:

\begin{lemma}
\label{lem:GV4-Sarkisov-link}
There exists a $\mathbb{T}_{3}\rtimes\mathbb{W}_{3}$-Sarkisov link
\begin{equation}
\xymatrix{ & \widetilde{X}_{20}\ar@{->}[dl]_{\gamma}\ar@{-->}[rr]^{\rho} &  & \overline{X}_{20}\ar@{->}[dr]^{\delta}\\
V_{4}\ar@{-->}[rrrr]^{\varphi} &&&& X_{24}}
\label{eq:V4-X24-link}
\end{equation}
where $\gamma$ is the blow-up of the six singular points of the threefold $V_{4}$,
the~map $\rho$ is a~composition of Atiyah flops in the proper transforms of the twelve $\mathbb{T}_{3}$-invariant lines in $V_{4}$,
and~$\delta$ is the composition of Kawamata blow-ups of the eight singular points of $X_{24}$.
\end{lemma}

Let $V_{2}$ be the hypersurface of degree $4$ in $\mathbb{P}(1,1,1,1,2)$ defined by the equation
\begin{equation}
\label{equation:toric-quartic-double-solid}
w^{2}-x_{1}x_{2}x_{3}x_{4}=0,
\end{equation}
where $x_{1}$, $x_{2}$, $x_{3}$ and $x_{4}$ are coordinates
of weight $1$, and $w$ is a coordinate of weight $2$.
We view $V_{2}$ as a toric variety for the torus $\mathbb{T}_{3}$
with open orbit $\mathbb{T}_{V_{2}}$ given by
$x_{1}x_{2}x_{3}x_{4}w\ne 0$.
The threefold $V_{2}$ has four $\mathbb{T}_{2}$-fixed points,
it has six $\mathbb{T}_{2}$-invariants irreducible curves, which are singular curves of the threefold $V_2$,
and it has four $\mathbb{T}_{2}$-invariants irreducible surfaces.

The rational map $\mathbb{P}(1,1,1,1,2)\dashrightarrow V_{6}$ given
by
$$
\big[x_{1}:x_{2}:x_{3}:x_{4}:w\big]\mapsto\Big(\big[x_{1}x_{2}:w\big],\big[x_{1}x_{3}:w\big],\big[x_{2}x_{3}:w\big]\Big)
$$
induces a $\mathbb{T}_{2}$-equivariant birational map $\Psi_{\infty}:V_{2}\dashrightarrow V_{6}$,
whose inverse is given by
\begin{equation}
\big(\big[u_{1}:v_{1}\big],\big[u_{2}:v_{2}\big],\big[u_{3}:v_{3}\big]\big)\mapsto\big[u_{1}u_{2}v_{3}:u_{1}u_{3}v_{2}:u_{2}u_{3}v_{1}:v_{1}v_{2}v_{3}:u_{1}u_{2}u_{3}v_{1}v_{2}v_{3}\big].\label{equation:toric-quartic-double-solid-inverse}
\end{equation}

With the Notation~\ref{notations:P1-P1-P1}, this birational map $\Psi_{\infty}$ fits in the following commutative diagram:
\begin{equation}
\label{equation:P1-P1-P1-S4-bad-link}
\xymatrix{&\widetilde{V}_{2}\ar@{->}[dl]_{\alpha_{\infty}}\ar@{->}[dr]^{\beta_{\infty}}\\
V_{2}\ar@{-->}[rr]^{\Psi_{\infty}} && V_{6}}
\end{equation}
where $\beta_{\infty}$ is the blow-up of the four points $(\infty,\infty,\infty)$, $(0,0,\infty)$, $(0,\infty,0)$ and $(\infty,0,0)$,
and $\alpha$ is the contraction of the proper transforms of the six $\mathbb{T}_{2}$-invariant
irreducible surfaces in $V_{6}$ onto the six singular curves of $V_{2}$.

Arguing as in the construction of $\Phi$ in Section~\ref{section:two-Sarkisov-links},
we see that
$$
\Psi_{\infty}^{-1}\overline{\mathbb{W}}_{2}^{\mathfrak{S}}\Psi_{\infty}=\mathbb{W}_{V_{2}},
$$
where we identified $\mathbb{W}_{V_6}=\mathbb{W}_{2}$.
Therefore, $\Phi_{\infty}$ is a birational $\mathbb{T}_{2}\rtimes\overline{\mathbb{W}}_{2}^{\mathfrak{S}}$-map.
The diagram \eqref{equation:P1-P1-P1-S4-bad-link} is a \emph{bad} $\mathbb{T}_{2}\rtimes\overline{\mathbb{W}}_{2}^{\mathfrak{S}}$-Sarkisov link.

The composition $\Psi_{\infty}^{-1}\circ\tau_{V_{6}}\circ\Psi_{\infty}$
is the biregular involution $\tau_{V_{2}}$ of $V_{2}$ defined by
$$
[x_{1}:x_{2}:x_{3}:x_{4}:w]\mapsto[x_{1}:x_{2}:x_{3}:x_{4}:-w].
$$
Viewing $\mathbb{P}^{3}$ as a toric variety for the torus $\mathbb{T}_{3}$
with open orbit $\mathbb{T}_{\mathbb{P}^{3}}$ given by $x_{1}x_{2}x_{3}x_{4}\ne0$,
the quotient morphism $V_{2}\to \mathbb{P}^{3}$ is equivariant
for the quotient morphism $q_{23}:\mathbb{T}_{2}\to\mathbb{T}_{3}$.
We can identify $\mathbb{W}_{\mathbb{P}^{3}}=\mathbb{W}_{V_{2}}=\overline{\mathbb{W}}_{3}^{\mathfrak{S}}$,
so that $G_{\mathbb{P}^{3}}=\mathbb{T}_{3}\rtimes\overline{\mathbb{W}}_{3}^{\mathfrak{S}}$.

It follows that the map $\Psi_{\infty}$ in \eqref{equation:P1-P1-P1-S4-bad-link}
induces a birational $\mathbb{T}_{3}\rtimes\overline{\mathbb{W}}_{3}^{\mathfrak{S}}$-map
$\psi_{\infty}\colon\mathbb{P}^{3}\dashrightarrow X_{24}$ given by
\begin{equation}
\label{eq:Bir-P3-X24}
\begin{array}{ccc}
\big[x_{1}:x_{2}:x_{3}:x_{4}\big] & \mapsto & \Big[x_{1}^{2}x_{2}^{2}x_{3}^{2}:x_{1}^{3}x_{2}x_{3}x_{4}:x_{1}^{2}x_{2}^{2}x_{3}x_{4}:x_{1}x_{2}^{3}x_{3}x_{4}:x_{1}^{2}x_{2}^{2}x_{4}^{2}:\\
 &  & x_{1}^{2}x_{2}x_{3}^{2}x_{4}:x_{1}x_{2}^{2}x_{3}^{2}x_{4}:x_{1}^{2}x_{2}x_{3}x_{4}^{2}:x_{1}x_{2}^{2}x_{3}x_{4}^{2}:\\
 &  & x_{1}x_{2}x_{3}^{3}x_{4}:x_{1}^{2}x_{3}^{2}x_{4}^{2}:x_{1}x_{2}x_{3}^{2}x_{4}^{2}:x_{2}^{2}x_{2}^{2}x_{4}^{2}:x_{1}x_{2}x_{3}x_{4}^{3}\Big],
\end{array}
\end{equation}
and we eventually obtain the following:

\begin{lemma}
\label{lem:GP3-Sarkisov-link}
There exists a $\mathbb{T}_{3}\rtimes\overline{\mathbb{W}}_{3}^{\mathfrak{S}}$-Sarkisov
link
\begin{equation}
\label{equation:X24-link-1}
\xymatrix{ & X_{22}\ar@{->}[dl]_{\gamma_{\infty}}\ar@{->}[rd]^{\delta_{\infty}}\\
\mathbb{P}^{3}\ar@{-->}[rr]^{\psi_{\infty}} &  & X_{24}}
\end{equation}
where $\delta_{\infty}$ is the composition of Kawamata blow-ups of four points in~$\mathrm{Sing}(X_{24})$
that form the $\overline{\mathbb{W}}_{3}^{\mathfrak{S}}$-orbit of the singular point $(\infty,\infty,\infty)$,
and $\gamma_{\infty}$ is the contraction of the~proper transforms of the six $\mathbb{T}_{3}$-invariant surfaces in $X_{24}$
to the six $\mathbb{T}_{3}$-invariant lines in $\mathbb{P}^{3}$.
\end{lemma}

Note that the birational map \eqref{eq:Bir-P3-X24} is defined by the linear system consisting
of all sextic surfaces that are singular along
the six $\mathbb{T}_{3}$-invariant lines in $\mathbb{P}^{3}$.
This recovers the original construction of the $\mathbb{T}_{3}\rtimes\overline{\mathbb{W}}_{3}^{\mathfrak{S}}$-Sarkisov
link \eqref{equation:X24-link-1} that is given in~\cite{CheltsovShramov2017}.

\begin{remark} Considering $\mathbb{P}^3$ as a toric variety for the torus $\mathbb{T}_3$ and considering the action of the group $\mathbb{W}_{\mathbb{P}^3}\cong\mathfrak{S}_4$ on the character lattice $M_3$ of $\mathbb{T}_3$, we see that $\mathbb{W}_{\mathbb{P}^3}=\overline{\mathbb{W}}_3^{\mathfrak{S}}$. Then there exists  the following $\mathbb{T}_3\rtimes\overline{\mathbb{W}}_3^{\mathfrak{S}}$-equivariant commutative diagram:
$$
\xymatrix{
\widetilde{V}_4\ar@{->}[dr]_{\beta}\ar@{->}[rr]^{\alpha}&&V_4\ar@{->}[rr]^{\sigma_{V_{4}}}&&V_4&&\widetilde{V}_4\ar@{->}[dl]^{\beta}\ar@{->}[ll]_{\alpha}\\
&\mathbb{P}^3\ar@{-->}[rrrr]&&&&\mathbb{P}^3&}
$$
where $\beta$ is the blow-up of the four $\mathbb{T}_{3}$-invariant points,
$\alpha$ is the contraction of the proper transforms of the six $\mathbb{T}_{3}$-invariant lines,
and the dashed arrow is the standard Cremona involution.
\end{remark}

\section{Proof of Theorem \ref{theorem:main} (infinite groups)}
\label{section:proof}

In this section, we give an alternative proof of Theorem~\ref{theorem:main}(3) in the case when the group $G$~is~infinite.
We~will~treat each of the threefolds $Y_{24}$,~$V_6$,~$X_{24}$,~$V_4$ and $\mathbb{P}^3$ in a separate subsection.

\subsection{Singular Fano threefold $Y_{24}$}
\label{subsection:Y-24-rigidity}
We use the notation introduced in Section~\ref{section:Y-24}.
Let $G_{Y_{24}}=\mathbb{T}_1\rtimes \mathbb{W}_{1}$,
let $\mathbb{W}$ be a subgroup in $\mathbb{W}_1$ that contains $\mathbb{W}_{1}^{\mathfrak{A}}$ (see Notation~\ref{notation:subgroups}),
and let $G=\mathbb{T}_1\rtimes\mathbb{W}\subset G_{Y_{24}}$.

\begin{lemma}
\label{lemma:Y-24}
The threefold $Y_{24}$ is $G$-birationally super-rigid.
\end{lemma}

\begin{proof}
Suppose that $Y_{24}$ is not $G$-birationally super-rigid.
Then, see for instance \cite[Theorem~3.3.1]{CheltsovShramov}, there exists a $G$-invariant mobile linear system $\mathcal{M}$ on the threefold $Y_{24}$ such that
the pair $(Y_{24},\lambda\mathcal{M})$ is not canonical,
where $\lambda$ is a positive rational number defined by
$$
\lambda\mathcal{M}\sim_{\mathbb{Q}}-K_{Y_{24}}.
$$

Let $Z$ be a $G$-irreducible center of non-canonical singularities of the log pair $(Y_{24},\lambda\mathcal{M})$.
Then, by Lemma~\ref{lemma:Y24-G-Fano}, we have one of the following possibilities:
\begin{enumerate}
\item $Z$ is the $G$-orbit of the singular point $(0,0,\infty,\infty)$,
\item $Z$ is the $G$-orbit of the smooth point $(0,\infty,\infty,\infty)$ or of the smooth point $(0,0,0,\infty)$,
\item $Z$ is the $G$-orbit of the curve $(0,0,1,\infty)$ or of the curve $(0,1,\infty,\infty)$.
\end{enumerate}
Let us show that none of these three cases is actually possible.

Let $S$ be the surface $(0,1,1,\infty)$.
Then $S\cong\mathbb{P}^1\times\mathbb{P}^1$, and the restriction $\lambda\mathcal{M}\vert_{S}$ is an effective $\mathbb{Q}$-linear system of bi-degree $(1,1)$.
Then $S$ contains the singular points $(0,0,\infty,\infty)$ and $(0,\infty,0,\infty)$,
it also contains the smooth points $(0,0,0,\infty)$ and $(0,\infty,\infty,\infty)$,
and it also contains the curves $(0,0,1,\infty)$, $(0,1,0,\infty)$, $(0,1,\infty,\infty)$ and  $(0,\infty,1,\infty)$.

If $Z$ is a curve, then
the multiplicity of the restriction $\lambda\mathcal{M}\vert_{S}$ at
the curve $Z$ is strictly larger than $1$ by Lemma~\ref{lemma:exercise}.
Clearly, this is impossible, since $\lambda\mathcal{M}\vert_{S}$ has bi-degree $(1,1)$.
Thus $Z$ must be zero dimensional.

Suppose that $Z$ is the $G$-orbit of the smooth point $(0,0,0,\infty)$ or of the smooth point $(0,\infty,\infty,\infty)$.
Denote this point by $P$. Then the tangent space $T_{P}(Y_{24})$ is an irreducible representation  of the stabilizer of the point $P$ in the group $G$. Thus, by~Lemma~\ref{lemma:mult-2}, we have
$$
\mathrm{mult}_{P}\big(\mathcal{M}\big)>\frac{2}{\lambda}.
$$
Let $C$ be a general curve in $S$ of bi-degree $(1,1)$ that passes through $P$.
Such curves span the whole surface $S$, so that $C$ is not contained in the base locus of the linear system $\mathcal{M}$.
Then, for a general surface $M\in\mathcal{M}$, we have
$$
\frac{2}{\lambda}=M\cdot C\geqslant \mathrm{mult}_{P}\big(\mathcal{M}\big)>\frac{2}{\lambda},
$$
which is absurd.

It thus remains to consider the case where $Z$ consists of singular points of the threefold $Y_{24}$. Let $\alpha\colon\widetilde{Y}_{24}\to Y_{24}$ be the blow-up of the points $(0,0,\infty,\infty)$ and $(0,\infty,0,\infty)$,
let $\widetilde{M}$ be the proper transform of a general surface in the linear system $\mathcal{M}$ on the threefold $\widetilde{Y}_{24}$,
let~$E_1$ and $E_2$ be the $\alpha$-exceptional surfaces.
Then
$$
\lambda\widetilde{M}\sim_{\mathbb{Q}}\alpha^*(-K_{Y_{24}})-m_1E_1-m_2E_2
$$
for some rational numbers $m_1$ and $m_2$. By~Lemma~\ref{lemma:Corti}, we have $m_1>1$ and $m_2>1$.
Now let $\widetilde{C}$ be the proper transform on $\widetilde{Y}_{24}$ of a general curve in $S$ of bi-degree $(1,1)$ that passes through
both points $(0,0,\infty,\infty)$ and $(0,\infty,0,\infty)$.
Then $\widetilde{C}\not\subset\widetilde{M}$, so that
$$
0\leqslant \lambda\widetilde{M}\cdot\widetilde{C}=\Big(\alpha^*(-K_{Y_{24}})-m_1E_1-m_2E_2\Big)\cdot \widetilde{C}=2-m_1-m_2<0,
$$
which is absurd. This completes the proof of the lemma.
\end{proof}

\begin{remark} Since the Fano threefold $Y_{24}$ is $G_{Y_{24}}$-birationally super-rigid, there is no $G_{Y_{24}}$-Sarkisov link starting at $Y_{24}$. But there are \emph{bad} $G_{Y_{24}}$-Sarkisov links that start at $Y_{24}$,
which implicitly appear in the proof of Lemma~\ref{lemma:Y-24}.
For example, blowing-up all singular points of the threefold $Y_{24}$,
we obtain the \emph{bad} $G_{Y_{24}}$-Sarkisov link \eqref{equation:bad-link-Y-24}.
\end{remark}

\subsection{Fano threefold $V_6=\mathbb{P}^1\times\mathbb{P}^1\times\mathbb{P}^1$}
\label{subsection:P1-P1-P1-rigidity}
We use the notation introduced in Section~\ref{subsection:P1-P1-P1}.
Let $G_{V_6}=\mathbb{T}_2\rtimes \mathbb{W}_{2}$,
let $\mathbb{W}$ be one of the subgroups in $\mathbb{W}_2$ that contains the group $\mathbb{W}_{2}^{\mathfrak{A}}$,
and let $G=\mathbb{T}_2\rtimes\mathbb{W}\subset G_{V_6}$.

\begin{lemma}
\label{lemma:P1-P1-P1}
The threefold $V_{6}$ is $G$-birationally super-rigid.
\end{lemma}

\begin{proof}
We may assume that $\mathbb{W}=\mathbb{W}_{2}^{\mathfrak{A}}$.
Suppose that $V_{6}$ is not $G$-birationally super-rigid.
Then  there exists a $G$-invariant mobile linear system $\mathcal{M}$ on $V_{6}$ such that $(V_{6},\lambda\mathcal{M})$ is not canonical,
where $\lambda$ is the positive rational number defined by $\lambda\mathcal{M}\sim_{\mathbb{Q}}-K_{V_{6}}$.

Let $Z$ be a $G$-irreducible center of non-canonical singularities of the log pair $(V_{6},\lambda\mathcal{M})$.
If $Z$ is a curve, we can assume that $Z$ is the $\mathbb{W}_{2}^{\mathfrak{A}}$-orbit of the $\mathbb{T}_2$-invariant curve $(0,0,1)$.
Otherwise, we can assume that $Z$ is the $\mathbb{W}_{2}^{\mathfrak{A}}$-orbit of the point $(0,0,0)$.

Let $S$ be the surface $(0,1,1)\subset V_{6}$.
Then $S\cong\mathbb{P}^1\times\mathbb{P}^1$, and $\lambda\mathcal{M}\vert_{S}$ is an effective $\mathbb{Q}$-linear system of bi-degree $(2,2)$.
If $Z$ is a curve, then
$$
S\cap Z=(0,0,1)\cup (0,\infty,1) \cup (0,1,0) \cup (0,1,\infty).
$$
We let $C_1=(0,0,1)$ and $C_2=(0,\infty,1)$. Note that $C_1\cap C_2=\varnothing$.
If $Z$ is a point, then
$$
S\cap Z=(0,0,0)\cup (0,\infty,\infty).
$$
We let $P_1=(0,0,0)$ and $P_2=(0,\infty,\infty)$.

If $Z$ is a curve, then it follows from Lemma~\ref{lemma:exercise} that
$$
\mathrm{mult}_{C_1}\big(\mathcal{M}\big)=\mathrm{mult}_{C_2}\big(\mathcal{M}\big)>\frac{1}{\lambda}
$$
This implies that the coefficient of these curves in the restriction $\lambda\mathcal{M}\vert_{S}$ is larger than $1$,
contradicting the fact that $\lambda\mathcal{M}\vert_{S}$ is of bi-degree $(2,2)$.
Thus $Z$ must be zero dimensional.

Since the stabilizer of the point $P_1$ in the group $G$ contains a subgroup $\mumu_3$ that permutes transitively the $\mathbb{T}_2$-invariant curves that pass through $P_1$, the tangent space $T_{P_1}(X)$ is an irreducible three-dimensional representation of the stabilizer of $P_1$.
Therefore, it follows from Lemma~\ref{lemma:mult-2} that
$$
\mathrm{mult}_{P_1}\big(\mathcal{M}\big)=\mathrm{mult}_{P_2}\big(\mathcal{M}\big)>\frac{2}{\lambda}.
$$
Let $C$ be a general curve in the surface $S$ of bi-degree $(1,1)$ that passes through $P_1$ and~$P_2$.
Such curves span the whole surface $S$, so that the curve $C$ is not contained in the base locus of the linear system $\mathcal{M}$. Thus, for a general surface $M\in\mathcal{M}$, we have
$$
4=\lambda M\cdot C\geqslant\lambda\Big(\mathrm{mult}_{P_1}(\mathcal{M})+\mathrm{mult}_{P_2}(\mathcal{M})\Big)>4,
$$
which is absurd. This completes the proof.
\end{proof}

\begin{remark} Since  threefold $V_6$ is $G_{V_6}$-birationally super-rigid, there is no $G_{V_6}$-Sarkisov link that starts at $V_6$.
However, there exist \emph{bad} $G_{V_6}$-Sarkisov links that start at $V_6$. For instance, blowing-up all $\mathbb{T}_2$-invariant points leads to the \emph{bad} $G_{V_6}$-Sarkisov link~\eqref{equation:P1-P1-P1-S4-bad-link}.
Likewise, the $G_{V_6}$-equivariant symbolic blow-up $\alpha$ of the union of all $\mathbb{T}_2$-invariant curves (see Example~\ref{ex:symbolic-BU}) also leads to a \emph{bad} $G_{V_6}$-Sarkisov link:
$$
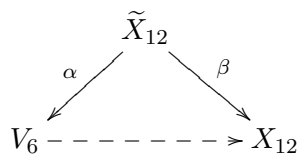
\xymatrix{
&\widetilde{X}_{12}\ar@{->}[dl]_{\alpha}\ar@{->}[dr]^{\beta}&\\%
V_6 \ar@{-->}[rr] && X_{12}}
$$
where $\beta$ is the contraction of the proper transforms of the $G_{V_6}$-invariant surfaces in $V_6$.
One~can show that $X_{12}$ is the canonical toric Fano threefold~\textnumero{9099} in \cite{grdb},
which can also be obtained as the~quotient of the singular Fano threefold $Y_{24}$, viewed as a toric variety for the action of the torus $\mathbb{T}_1$, by an involution that fixes only $\mathbb{T}_1$-invariant points.
\end{remark}

\subsection{Singular Fano threefolds $V_4$ and $X_{24}$}
\label{section:V4-X24}
We now treat the threefolds $V_4$ and~$X_{24}$.
We use the notation of Sections~\ref{subsection:Fano-Enriques} and \ref{section:two-Sarkisov-links},
and we identify $G_{V_4}=G_{X_{24}}=\mathbb{T}_3\rtimes\mathbb{W}_3$.
By Section~\ref{section:two-Sarkisov-links}, we have a $\mathbb{T}_3\rtimes\mathbb{W}_3$-Sarkisov link
\begin{align*}
\xymatrix{
&\widetilde{X}_{20}\ar@{->}[dl]_{\gamma}\ar@{-->}[rr]^{\rho}&&\overline{X}_{20}\ar@{->}[dr]^{\delta}\\%
V_{4} \ar@{-->}[rrrr]^{\varphi} &&&& X_{24}}
\end{align*}
where $\gamma$ is the blow-up of the eight singular points of $V_{4}$,
the~map $\rho$ is a~composition of Atiyah flops in the proper transforms of the twelve $\mathbb{T}_3$-invariant lines in $V_4$,
and $\delta$ is the composition of Kawamata blow-ups of the eight singular points of $X_{24}$.

Let $\mathbb{W}$ be a subgroup of $\mathbb{W}_3$ that contains $\mathbb{W}_3^{\mathfrak{A}}$
such that either
$\mathbb{W}=\mathbb{W}_3^{\mathfrak{A}}$, or $\mathbb{W}$ contains the involution $\sigma$
(see Notation~\ref{notation:subgroups}), and let $G=\mathbb{T}_3\rtimes\mathbb{W}$. The proof of the following lemma is straightforward.

\begin{lemma}
\label{lemma:X24-V4}
Let $\mathcal{M}_{X_{24}}$ be a $G$-invariant mobile linear system on $X_{24}$,
and let $\mathcal{M}_{V_4}$ be its proper transform on $V_4$ via $\varphi$.
Then the following assertions hold:
\begin{enumerate}
\item There are $k\in\mathbb{Z}_{\geqslant 0}$ and $m\in\frac{1}{2}\mathbb{Z}_{\geqslant 0}$
such that $\mathcal{M}_{X_{24}}\sim_{\mathbb{Q}} -kK_{X_{24}}$ and
$$
\delta^{-1}_*\big(\mathcal{M}_{X_{24}}\big)\sim_{\mathbb{Q}}\delta^*\big(-kK_{X_{24}}\big)-m\sum _{i=1}^8 F_i,
$$
where each $F_i$ is a $\delta$-exceptional surface.

\item There are $n\in\mathbb{Z}_{\geqslant 0}$ and $m^\prime\in\mathbb{Z}_{\geqslant 0}$ such that $\mathcal{M}_{V_4}\sim n\mathcal{H}_{V_4}$ and
$$
\gamma^{-1}_*\big(\mathcal{M}_{V_4}\big)\sim_{\mathbb{Q}}\gamma^*\big(n\mathcal{H}_{V_4}\big)-m^\prime\sum_{i=1}^6 E_i,
$$
where $\mathcal{H}_{V_4}$ is a hyperplane section of $V_4$,
and each $E_i$ is a $\gamma$-exceptional surface.
\item One has $n=3k-2m$ and $k=n-m^\prime$.
\end{enumerate}
\end{lemma}

Now we are ready to prove

\begin{proposition}
\label{theorem:main-V4-solid}
One has $\mathcal{P}_G(V_4)=\{V_4,X_{24}\}$.
\end{proposition}

\begin{proof}
Let $\chi\colon X_{24}\dasharrow Y$ be a $G$-equivariant birational map  such that
$Y$~is~a~threefold with terminal singularities, and there exists a $G$-equivariant morphism $f\colon Y\to Z$ that is a $G$-Mori fiber space.
Fix a~sufficiently large positive integer $n\gg 0$. Let $D_Z$ be a sufficiently general very ample divisor on $Z$,
let $\mathcal{M}_{Y}=|-nK_Y+f^*(D_Z)|$, let
$$
\mathcal{M}_{X_{24}}=(\chi)_*^{-1}\big(\mathcal{M}_Y\big)
$$
and let
$$
\mathcal{M}_{V_{4}}=(\chi\circ\varphi)_*^{-1}\big(\mathcal{M}_Y\big).
$$
Then $\mathcal{M}_{X_{24}}$ and $\mathcal{M}_{V_4}$ are $G$-invariant mobile linear systems on $X_{24}$ and $V_{4}$, respectively.
Let $k$ and $n$ be the non-negative integers such that
$\mathcal{M}_{X_{24}}\sim_{\mathbb{Q}} -kK_{X_{24}}$ and $\mathcal{M}_{V_4}\sim n\mathcal{H}_{V_4}$.
If the log pair $(X_{24},\frac{1}{k}\mathcal{M}_{X_{24}})$ has canonical singularities,
then it~follows from \emph{the~Noether--Fano inequality} that $\chi$ is an isomorphism,
see \cite[Theorems~3.2.1~and~3.2.6]{CheltsovShramov} or \cite{Co95}.
Similarly, if the log pair $(V_{4},\frac{2}{n}\mathcal{M}_{V_{4}})$ has canonical singularities,
then the birational map $\chi\circ\varphi$ is an isomorphism.
Therefore, to prove the required assertion,
it is enough to show that either $(X_{24},\frac{1}{k}\mathcal{M}_{X_{24}})$
or $(V_{4},\frac{2}{n}\mathcal{M}_{V_{4}})$ has canonical singularities.

Suppose that the singularities of the log pair $(X_{24},\frac{1}{k}\mathcal{M}_{X_{24}})$ are worse than canonical.
Then, using Lemma~\ref{lemma:Kawamata} and Lemma~\ref{lemma:X24-V4}, we obtain the inequality $n<2k$.
Let us show then that the log pair $(V_{4},\frac{2}{n}\mathcal{M}_{V_{4}})$ has canonical singularities.

By construction of $\varphi$, if the singularities of the log pair $(V_{4},\frac{2}{n}\mathcal{M}_{V_{4}})$ are not canonical, then
the union of its centers of non-canonical singularities is either the union of all singular points of the threefold $V_4$, or
a $\mathbb{W}$-orbit of $\mathbb{T}_3$-invariant curves, which are lines in $\mathbb{P}^5$.

In the first case, using Lemma~\ref{lemma:Corti} and Lemma~\ref{lemma:X24-V4}, we get $n>2k$,
which is impossible, since we already proved that $n<2k$.

The twelve $\mathbb{T}_3$-invariant lines in $V_4$ form a unique $\mathbb{W}$-irreducible curve.
Suppose that all of them are centers of non-canonical singularities of the log pair $(V_{4},\frac{2}{n}\mathcal{M}_{V_{4}})$.
Then
\begin{equation}
\label{equation:V4-mult}
\mathrm{mult}_{L}\big(\mathcal{M}_{V_{4}}\big)>\frac{n}{2}
\end{equation}
for each such line $L$ by Lemma~\ref{lemma:exercise}.
On the other hand, each of the eight $\mathbb{T}_3$-invariant planes in $V_4$ contains three $\mathbb{T}_3$-invariant lines.
Thus, restricting $\mathcal{M}_{V_{4}}$ on one such plane, we obtain a contradiction to \eqref{equation:V4-mult}.
\end{proof}

\subsection{Fano threefolds $\mathbb{P}^3$ and $X_{24}$}
\label{section:P3}
Finally, we deal with the Fano threefolds $\mathbb{P}^3$ and~$X_{24}$. For $X_{24}$, we use the same notation as in Section~\ref{section:V4-X24}. As~in~Section~\ref{section:two-Sarkisov-links}, we view $\mathbb{P}^3$ as a toric variety for the torus $\mathbb{T}_3$,
and we identify $\mathbb{W}_{\mathbb{P}^{3}}=\overline{\mathbb{W}}_{3}^{\mathfrak{S}}$
and $G_{\mathbb{P}^{3}}=\mathbb{T}_{3}\rtimes\overline{\mathbb{W}}_{3}^{\mathfrak{S}}$.
In Section~\ref{section:two-Sarkisov-links}, we constructed the following $\mathbb{T}_{3}\rtimes\overline{\mathbb{W}}_{3}^{\mathfrak{S}}$-Sarkisov link:
\begin{equation*}
\xymatrix{
&X_{22}\ar@{->}[dl]_{\gamma_{\infty}}\ar@{->}[rd]^{\delta_\infty}&\\%
\mathbb{P}^3\ar@{-->}[rr]^{\psi_\infty} &&X_{24} }
\end{equation*}
where $\delta_\infty$ is a composition of Kawamata blow-ups of the four singular points of $X_{24}$
that form the $\overline{\mathbb{W}}_{3}^{\mathfrak{S}}$-orbit of the point $(\infty,\infty,\infty)$ in $X_{24}$,
and $\gamma_{\infty}$ is the contraction of the proper transforms of the six $\mathbb{T}_3$-invariant surfaces in $X_{24}$
to the six $\mathbb{T}_3$-invariant lines in $\mathbb{P}^3$.

Recall subsection \ref{subsection:Fano-Enriques} that the regular involution $\sigma_{X_{24}}$ of $X_{24}=V_6/\tau_{V_6}$ induced by the involution $\sigma_{V_{6}}=\upsilon\times \upsilon\times\upsilon$ (see Notation \ref{notations:P1-P1-P1}) commutes with the action of $\overline{\mathbb{W}}_{3}^{\mathfrak{S}}$. We thus obtain a second birational $\mathbb{T}_{3}\rtimes\overline{\mathbb{W}}_{3}^{\mathfrak{S}}$-map $\psi_0=\sigma_{X_{24}}\circ\psi_\infty\colon\mathbb{P}^3\dashrightarrow X_{24}$.
Note that $\sigma_{X_{24}}(\infty,\infty,\infty)=(0,0,0)$.
Consequently, we have a second $\mathbb{T}\rtimes\overline{\mathbb{W}}_3^{\mathfrak{S}}$-Sarkisov link
\begin{equation*}
\xymatrix{
&X_{22}\ar@{->}[dl]_{\gamma_0}\ar@{->}[rd]^{\delta_0}&\\%
\mathbb{P}^3\ar@{-->}[rr]^{\psi_0} &&X_{24} }
\end{equation*}
where $\delta_0$ is a composition of Kawamata blow-ups of the four singular points of $X_{24}$
which form the $\overline{\mathbb{W}}_{3}^{\mathfrak{S}}$-orbit of the point $(0,0,0)$,
and $\gamma_0$ is the contraction of the proper transforms of the six $\mathbb{T}_3$-invariant surfaces in $X_{24}$
to the six $\mathbb{T}_3$-invariant lines in $\mathbb{P}^3$.

\begin{remark}
\label{remak:Cremona-X-24}
Using \eqref{eq:Embed-X24-P13} and \eqref{eq:Bir-P3-X24},
one can show that $\psi^{-1}_\infty\circ\sigma_{X_{24}}\circ \psi_\infty$ is equal to the~standard Cremona involution
$\sigma_{\mathbb{P}^3}\colon\mathbb{P}^3\dasharrow\mathbb{P}^3$ defined by
\begin{equation}
\label{equation:Cremona}
[x_1:x_2:x_3:x_4]\mapsto[x_2x_3x_4:x_1x_3x_4:x_1x_2x_4:x_1x_2x_3].
\end{equation}
In other words, we have a commutative diagram of birational $\mathbb{T}\rtimes \overline{\mathbb{W}}_3^{\mathfrak{S}}$-maps
$$
\xymatrix{ \mathbb{P}^3 \ar@{-->}[rrrr]^{\psi_\infty} \ar@{-->}[drrrr]^{\psi_0} \ar@{-->}[d]_{\sigma_{\mathbb{P}^3}} &&&& X_{24} \ar[d]^{\sigma_{X_{24}}} \\
\mathbb{P}^3 \ar@{-->}[rrrr]^{\psi_\infty} &&&& X_{24}. }
$$
\end{remark}

Let $E$ be the sum of all $\gamma$-exceptional surfaces,
let $F_0$ be the sum of all $\delta_0$-exceptional surfaces,
let $F_\infty$ be the sum of all $\delta_\infty$-exceptional surfaces,
and let  $\mathbb{W}$ be either $\mathbb{W}_3^{\mathfrak{A}}$~or~$\overline{\mathbb{W}}_3^{\mathfrak{S}}$.
We also let $G=\mathbb{T}_3\rtimes \mathbb{W}$. The next lemma is straightforward.

\begin{lemma}
\label{lemma:X24-P3}
Let $\mathcal{M}_{X_{24}}$ be a $G$-invariant mobile linear system on the Fano threefold~$X_{24}$,
let $\mathcal{M}_{X_{22},\infty}$ and $\mathcal{M}_{X_{22},0}$ be its proper transforms on $X_{22}$ via $\delta_\infty$ and $\delta_0$, respectively,
and~let~$\mathcal{M}_{\mathbb{P}^3,\infty}$ and $\mathcal{M}_{\mathbb{P}^3,0}$ be its proper transforms on $\mathbb{P}^3$ via $\psi_\infty$ and $\psi_0$, respectively.
Furthermore,~let $k\in\frac{1}{2}\mathbb{Z}$ and let $n_\infty$ and $n_0$ be integers such that
$$
\left\{\aligned
&\mathcal{M}_{X_{24}}\sim_{\mathbb{Q}} -kK_{X_{24}},\\
&\mathcal{M}_{\mathbb{P}^3,\infty}\sim_{\mathbb{Q}} n_\infty H,\\
&\mathcal{M}_{\mathbb{P}^3,0}\sim_{\mathbb{Q}} n_0H.\\
\endaligned
\right.
$$
where $H$ is a hyperplane in $\mathbb{P}^3$. Then the following assertions hold:
\begin{enumerate}
\item There are $m_0$ and $m_\infty$ in $\frac{1}{2}\mathbb{Z}_{\geqslant 0}$   such that
$$
\left\{\aligned
&\mathcal{M}_{X_{22},0} \sim_{\mathbb{Q}} \delta_0^*\big(-kK_{X_{24}}\big)-m_0F_0,  \\
&\mathcal{M}_{X_{22},\infty} \sim_{\mathbb{Q}} \delta_\infty^*\big(-kK_{X_{24}}\big)-m_\infty F_\infty.
\endaligned
\right.
$$

\item There are $m^\prime_0$ and $m^\prime_\infty$ in $\mathbb{Z}_{\geqslant 0}$ such that
$$
\left\{\aligned
&\mathcal{M}_{X_{22},0} \sim_{\mathbb{Q}} \gamma_0^*\big(n_0\mathcal{H})-m^\prime_0 E, \\
&\mathcal{M}_{X_{22},\infty} \sim_{\mathbb{Q}} \gamma_\infty^*\big(n_\infty\mathcal{H})-m^\prime_\infty E.
\endaligned
\right.
$$

\item Furthermore, one has
$$
\left\{\aligned
&n_0=6k-4m_0,\\
&n_\infty=6k-4m_\infty,\\
&n_0=3n_\infty-4m^\prime_\infty,\\
&k=\frac{n_0}{2}-m^\prime_0,\\
&k=\frac{n_\infty}{2}-m^\prime_\infty.
\endaligned
\right.
$$
\end{enumerate}
\end{lemma}

Now we are ready to prove:

\begin{proposition}
\label{proposition;P3-X24-solid}
One has $\mathcal{P}_G(\mathbb{P}^3)=\{\mathbb{P}^3,X_{24}\}$.
\end{proposition}

\begin{proof}
Let $\mathcal{M}_{X_{24}}$ be a $G$-invariant mobile linear system on the threefold $X_{24}$. With the notation of Lemma~\ref{lemma:X24-P3},
if the log pair $(X_{24},\frac{1}{k}\mathcal{M}_{X_{24}})$ does not have canonical singularities, then, combining Lemmas~\ref{lemma:Kawamata} and \ref{lemma:X24-P3}, we obtain that
\begin{itemize}
\item either $(X_{24},\frac{1}{k}\mathcal{M}_{X_{24}})$ is not canonical at the point $(0,0,0)$ and $n_0<4k$,
\item or $(X_{24},\frac{1}{k}\mathcal{M}_{X_{24}})$ is not canonical at the point $(\infty,\infty,\infty)$ and $n_\infty<4k$.
\end{itemize}

If $(\mathbb{P}^3,\frac{4}{n_\infty}\mathcal{M}_{\mathbb{P}^3_{\infty}})$ does not have canonical singularities,
we let $Z$ be its $G$-irreducible center of non-canonical singularities.
In this case, one of the following cases holds:
\begin{itemize}
\item $Z$ is the $G_{\mathbb{P}^3}$-irreducible curve and $k<\frac{n_\infty}{4}$;
\item $Z$ is the $G_{\mathbb{P}^3}$-orbit of length $4$ and $n_0<n_\infty$.
\end{itemize}
Indeed, if $Z$ is the $G_{\mathbb{P}^3}$-irreducible curve, then
$m_\infty^\prime>\tfrac{n_\infty}{4}$ by Lemma~\ref{lemma:exercise}, so that
$$
k=\frac{n_\infty}{2}-m^\prime_\infty<\frac{n_\infty}{4},
$$
by Lemma~\ref{lemma:X24-P3}.
Similarly, if $Z$ is the $G_{\mathbb{P}^3}$-orbit of length $4$, then
$$
m^\prime_\infty>\tfrac{n_\infty}{2}
$$
by Lemma \ref{lemma:mult-2}, because the tangent space $T_{P}(\mathbb{P}^3)$ at a point $P\in Z$  is an irreducible representation of the stabilizer of the point $P$ in the group $G$.
Thus, in this case, we have
$$
n_0=3n_\infty-4m^\prime_\infty<n_\infty
$$
by Lemma~\ref{lemma:X24-P3}.

Now, we let $q$ be the smallest number among $\frac{n_\infty}{4}$, $\frac{n_0}{4}$ and~$k$.
Without loss of generality, we may assume that
$$
q=\mathrm{min}\Big\{\frac{n_\infty}{4},k\Big\}.
$$
In view of the above alternatives, we obtain the following:
\begin{itemize}
\item if $q=\frac{n_\infty}{4}$, then $(\mathbb{P}^3,\frac{4}{n_\infty}\mathcal{M}_{\mathbb{P}^3,\infty})$ has canonical singularities;
\item if $q=k$, then $(X_{24},\frac{1}{k}\mathcal{M}_{X_{24}})$ has canonical singularities.
\end{itemize}
Now, arguing as in the proof of Proposition~\ref{theorem:main-V4-solid},
we deduce that $\mathbb{P}^3$ and $X_{24}$ are the only $G$-Mori fibre spaces $G$-birational to $\mathbb{P}^3$.
\end{proof}

\section{Proof of Theorem~\ref{theorem:main} (finite groups)}
\label{section:approximation}

All assertions of Theorem~\ref{theorem:main} follow from the results of Sections~\ref{section:toric} and \ref{section:Fano-threefolds} except
for the~part (3), which has been already proved in Sections~\ref{section:toric} and \ref{section:proof} for infinite groups.
The aim of this section is to prove Theorem~\ref{theorem:main}(3) for finite groups.
To do this, we need some results on finite subgroups of the groups $\mathbb{T}_1\rtimes\mathbb{W}_1$,
$\mathbb{T}_2\rtimes\mathbb{W}_2$ and $\mathbb{T}_3\rtimes \mathbb{W}_3$.

\subsection{Finite subgroups}
We use the notation of Section~\ref{section:Fano-threefolds}.

\begin{lemma}
\label{lemma:Adrien}
Let $\mathbb{W}$ be a subgroup of the finite group $\mathbb{W}_2$ that contains the group $\mathbb{W}_2^{\mathfrak{A}}$,
and let $G$ be a $\mathbb{W}$-invariant finite subgroup of $\mathbb{T}_2$. Then there exists $n\in\mathbb{N}$ such that one of the following three possibilities holds:
\begin{enumerate}
\item $G\cong\mumu_n^3$;
\item $n$ is even and $G\cong\mumu_{n}^2\times\mumu_{\frac{n}{2}}$;
\item $n$ is even and $G\cong\mumu_{n}\times\mumu_{\frac{n}{2}}^2$.
\end{enumerate}
\end{lemma}

\begin{proof}
Let $n$ be the maximal order of elements in $G$, and let $h$ be an~element in $G$ that has maximal order.
Then the order of every element of $G$ divides $n$. This implies that~$G\subseteq\mumu_{n}^{3}$.
Thus, there are positive integers $a$, $b$, $c$ such that $\mathrm{gcd}(a,b,c,n)=1$ and
$$
h=\big(\epsilon^a,\epsilon^b,\epsilon^c\big)
$$
for some primitive $n$th root of unity $\epsilon$.

With respect to the basis $f_{1}$, $f_{2}$, $f_{3}$ of the lattice $M_2$, the subgroup $\mathbb{W}_{2}^{\mathfrak{A}}\subset\mathrm{GL}_{3}(\mathbb{Z})$ is generated by permutation matrices and the matrices $$A=\left(\begin{array}{ccc}
-1 & 0 & 0\\
0 & -1 & 0\\
0 & 0 & 1
\end{array}\right)\quad\textrm{and}\quad B=\left(\begin{array}{ccc}
-1 & 0 & 0\\
0 & 1 & 0\\
0 & 0 & -1
\end{array}\right).$$

Applying cyclic permutations of order $3$ to $h$, we see that
$$
(\epsilon^c,\epsilon^a,\epsilon^b)\in G \ni (\epsilon^b,\epsilon^c,\epsilon^a),
$$
so that the group $G$ contains an element of the form $(\epsilon,\epsilon^\beta,\epsilon^\gamma)$ for some integers $\beta$ and~$\gamma$.
Thus, we may assume that $a=1$.
Applying $AB$ to the element $h$, we see that
$$
h^\prime=(\epsilon,\epsilon^{-b},\epsilon^{-c})\in G.
$$
It follows that $hh^\prime=(\epsilon^2,1,1)$
and its transforms $(1,\epsilon^2,1)$ and $(1,1,\epsilon^2)$ by permutations matrices in $\mathbb{W}$ are also contained in the group $G$.

Let $h_2=(1,\epsilon^2,1)$ and $h_3=(1,1,\epsilon^2)$.
If $n$ is odd then we can replace $h$ by
$$
hh_2^\beta h_3^\gamma=(\epsilon,1,1),
$$
where $\beta$ and $\gamma$ are integers such that $\beta b\equiv 1\ \mathrm{mod}\ n$ and $\gamma c\equiv 1\ \mathrm{mod}\ n$.
If $n$~is~even, we can replace $h$ by $hh_2^\beta h_3^\gamma\in G$
for $\beta=-\lfloor\frac{b}{2}\rfloor$ and $\gamma=-\lfloor\frac{c}{2}\rfloor$.
Therefore, we may assume that one of the following three cases holds:
$(b,c)=(0,0)$, $(b,c)=(1,0)$ and $(b,c)=(1,1)$.

If $(b,c)=(0,0)$, then $h=(\epsilon,1,1)$ from which it follows that $G=\mumu_n^3$.

If $(b,c)=(1,0)$, then $n$ is even and $h=(\epsilon,\epsilon,1)$.
In this case, applying permutations, we get that $(1,\epsilon,\epsilon)\in G$ and hence that $G$ contains the subgroup
$$
G^\prime=\langle (\epsilon,\epsilon,1),(1,\epsilon,\epsilon), (1,1,\epsilon^2)\rangle \cong\mumu_n^2\times\mumu_{\frac{n}{2}}.
$$
Since $G\subset \mumu_n^3$, it follows that either $G=G^\prime$ or $G\cong\mumu_n^3$.

Finally, if $(b,c)=(1,1)$, then $G$ contains the subgroup
$$
G^\prime=\langle (\epsilon,\epsilon,\epsilon) ,(1,\epsilon^2,1),(1,1,\epsilon^2)\rangle \cong \mumu_n\times\mumu_{\frac{n}{2}}^2,
$$
and hence either $G=G^\prime$ or $G\cong\mumu_n^3$, which completes the proof.
\end{proof}

\begin{corollary}
\label{corollary:Y-24}
Let $\mathbb{W}$ be a subgroup of the finite group $\mathbb{W}_1$ that contains the group $\mathbb{W}_1^{\mathfrak{A}}$,
and let $G$ be a $\mathbb{W}$-invariant finite subgroup of $\mathbb{T}_1$. Then there exists $n\in\mathbb{N}$ such that one of the following five possibilities holds:
\begin{enumerate}
\item $G\cong\mumu_n^3$;
\item $n$ is even and $G\cong\mumu_{n}^2\times\mumu_{\frac{n}{2}}$;
\item $n$ is even and $G\cong\mumu_{n}\times\mumu_{\frac{n}{2}}^2$;
\item $n$ is divisible by $4$ and $G\cong\mumu_{n}\times\mumu_{\frac{n}{2}}\times\mumu_{\frac{n}{4}}$;
\item $n$ is divisible by $4$ and $G\cong\mumu_{n}\times\mumu_{\frac{n}{4}}^2$.
\end{enumerate}
\end{corollary}

\begin{proof}
Let $q_{12}\colon\mathbb{T}_{1}\to\mathbb{T}_{2}$,
be the quotient map that corresponds to the inclusion $M_2\hookrightarrow M_1$ described in Section~\ref{section:Fano-threefolds}.
Then $q_{12}$ is given by
$$
\big(\tone{t}_{1},\tone{t}_{2},\tone{t}_{3}\big)\mapsto\big(\tone{t}_{1}\tone{t}_{2}, \tone{t}_{1}\tone{t}_{3}, \tone{t}_{2}\tone{t}_{3}\big),
$$
and its kernel consists of two elements $\pm (1,1,1)$. Let $\overline{G}$ be the image of $G$  by the map~$q_{12}$.
Then either $G\cong\overline{G}$ or there exists an exact sequence of groups
$$
1\longrightarrow\mumu_2\longrightarrow G\longrightarrow\overline{G}\longrightarrow 1.
$$
Since the $\mathbb{W}_2$-action on $M_2$ is induced by the restriction of the $\mathbb{W}_1$-action on $M_1$,
we see that $\overline{G}$ is normalized by the action of the subgroup $\mathbb{W}_2^{\mathfrak{A}}$. By Lemma~\ref{lemma:Adrien}, we obtain that
\begin{enumerate}
\item either $\overline{G}\cong\mumu_n^3$,
\item or $n$ is even and $\overline{G}\cong\mumu_{n}^2\times\mumu_{\frac{n}{2}}$,
\item or $n$ is even and $\overline{G}\cong\mumu_{n}\times\mumu_{\frac{n}{2}}^2$.
\end{enumerate}
This immediately implies the result.
\end{proof}

\begin{corollary}[{cf. \cite{Flannery}}]
\label{lemma:Sakovich}
Let $\mathbb{W}$ be a subgroup of the finite group $\mathbb{W}_3$ that contains $\mathbb{W}_3^{\mathfrak{A}}$,
and let $G$ be a $\mathbb{W}$-invariant finite subgroup of $\mathbb{T}_3$. Then there exists $n\in\mathbb{N}$ such that one of the following three possibilities holds:
\begin{enumerate}
\item $G\cong\mumu_n^3$;
\item $n$ is even and $G\cong\mumu_{n}^2\times\mumu_{\frac{n}{2}}$;
\item $n$ is divisible by $4$ and $G\cong\mumu_{n}^2\times\mumu_{\frac{n}{4}}$.
\end{enumerate}
\end{corollary}

\begin{proof}
Let $q_{23}\colon\mathbb{T}_{2}\to\mathbb{T}_{3}$ be the quotient by
the involution $(\ttwo{t}_{1},\ttwo{t}_{2},\ttwo{t}_{3})\mapsto(-\ttwo{t}_{1},-\ttwo{t}_{2},-\ttwo{t}_{3})$,
which corresponds to the inclusion of lattices $M_3\hookrightarrow M_2$ described in Section~\ref{section:Fano-threefolds},
and let $\widehat{G}$ be the preimage of the group $G$ via $q_{23}$.
Then $\widehat{G}$ is a finite subgroup in $\mathbb{T}_2$, which is normalized by the group $\mathbb{W}_2^{\mathfrak{A}}$.
Since $\widehat{G}$ contains $\pm (1,1,1)$, we see that $|\widehat{G}|$ is even.

It follows from the proof of Lemma~\ref{lemma:Adrien} that there exist an integer $m\in 2\mathbb{N}$ and a~primitive $m$-th root of unity $\epsilon$
such that $\widehat{G}$ is one of the following subgroups:
\begin{enumerate}
\item[(1)] $\langle (\epsilon,1,1),(1,\epsilon,1),(1,1,\epsilon)\rangle\cong\mumu_m^3$,
\item[(2)] $\langle (\epsilon,\epsilon,1),(1,\epsilon,\epsilon),(1,1,\epsilon^2)\rangle\cong\mumu_m^2\times\mumu_{\frac{m}{2}}$,
\item[(3)] $\langle (\epsilon,\epsilon,\epsilon),(1,\epsilon^2,1),(1,1,\epsilon^2)\rangle\cong\mumu_m\times\mumu^2_{\frac{m}{2}}$,
\end{enumerate}
Thus, in the first case, we have
$$
G=\langle (\epsilon,1,\epsilon),(\epsilon,\epsilon^{-1},1),(1,\epsilon^{-1},\epsilon)\rangle=\langle(\epsilon,1,\epsilon),(1,\epsilon,\epsilon),(1,\epsilon^{2},1)\rangle\cong\mumu_m^2\times\mumu_{\frac{m}{2}}.
$$
Similarly, in the second case, we have
$$
G=\langle (\epsilon^2,\epsilon^{-1},\epsilon),(\epsilon,\epsilon^{-2},\epsilon),(1,\epsilon^{-2},\epsilon^2)\rangle=\langle (\epsilon,1,\epsilon^{-1}),(1,\epsilon,\epsilon),(1,1,\epsilon^{4})\rangle=\mumu_m^2\times\mumu_{\frac{m}{4}},
$$
where $n$ is divisible by $4$, because $(-1,-1,-1)\in\widehat{G}$.
Finally, in the third case, we have
$$
G=\langle (\epsilon^2,\epsilon^{-2},\epsilon^2),(\epsilon^2,\epsilon^2,1),(1,\epsilon^{-2},\epsilon^{2})\rangle=\langle (\epsilon^2,1,1),(1,\epsilon^2,1),(1,1,\epsilon^{2})\rangle=\mumu_{\frac{m}{2}}^3.
$$
This completes the proof of the corollary.
\end{proof}

\subsection{The proof of Theorem~\ref{theorem:main}(3) for finite groups}
\label{subsection:the-proof}
Let $X$ be one of the threefolds $V_6=\mathbb{P}^1\times\mathbb{P}^1\times\mathbb{P}^1$, $X_{24}$, $Y_{24}$, $V_4$ or $\mathbb{P}^3$,
let~$\mathbb{T}$~be~a~maximal torus in $\mathrm{Aut}(X)$, and let $G_X$ be its normalizer in $\mathrm{Aut}(X)$.
Using the split exact sequence of groups
$$
\xymatrix{
1\ar@{->}[r]&\mathbb{T}\ar@{->}[r]& G_{X}\ar@{->}[r]^{\nu_X}&\mathbb{W}_X\ar@{->}[r]& 1,}
$$
we consider the Weyl group $\mathbb{W}_X$ as a subgroup of the group $G_X$.

Let~$G$ be a finite subgroup of the group $G_{X}$, let $\mathbb{W}=\nu_X(G)$, and let $\Gamma=\mathbb{T}\cap G$.
Suppose that the group $\mathbb{W}$ contains the unique subgroup in $\mathbb{W}_X$ that is isomorphic to $\mathfrak{A}_4$.
Then~Lemma~\ref{lemma:Adrien} and Corollaries~\ref{corollary:Y-24} and \ref{lemma:Sakovich} imply the following:

\begin{corollary}
\label{corollary:order-rank}
The group $\Gamma$ contains a subgroup $\Gamma^\prime\cong\mumu_n^3$ such that $|\Gamma^\prime:\Gamma|\leqslant 16$.
\end{corollary}

We have $G_{X}=\langle\mathbb{T},\mathbb{W}_X\rangle\cong\mathbb{T}\rtimes\mathbb{W}_X$ and $G=\langle \Gamma,\mathbb{W}\rangle\cong \Gamma\rtimes \mathbb{W}$.
Note that
$$
\mathrm{rk}\Big(\mathrm{Cl}(X)^G\Big)=\mathrm{rk}\Big(\mathrm{Cl}(X)^{\mathbb{W}}\Big),
$$
so that $X$ is $G$-minimal if and only if it is $\mathbb{W}$-minimal.
Now we suppose that $X$ is $\mathbb{W}$-minimal.
To prove Theorem~\ref{theorem:main}(3), we have to show that $X$ is $G$-solid if $|G|\geqslant 32\cdot 24^4$.

Suppose that $|\Gamma|\geqslant 16\cdot 24^3$. Note that this inequality follows from $|G|\geqslant 32\cdot 24^4$.
By~Corollary~\ref{corollary:order-rank}, the group $\Gamma$ contains a subgroup that is isomorphic to $\mumu_n^3$ for $n\geqslant 24$.
Let us prove that $X$ is $G$-solid.

Let $\mathbf{G}=\langle\mathbb{T}_X,\mathbb{W}\rangle\cong\mathbb{T}_X\rtimes\mathbb{W}$.
In Section~\ref{section:proof}, we proved that the threefold $X$ is $\mathbf{G}$-solid.
Moreover, this proof implies that $X$ is $G$-solid provided that the following condition is satisfied:
\begin{itemize}
\item[($\bigstar$)] For every non-empty $G$-invariant mobile linear system $\mathcal{M}$ on the threefold $X$,
all~non-canonical centers (if any) of the mobile log pair $(X,\lambda\mathcal{M})$ are $\mathbb{T}_X$-invariant,
where $\lambda$ is a positive rational number such that
$$
\lambda\mathcal{M}\sim_{\mathbb{Q}} -K_X.
$$
Moreover, if a $\mathbb{T}_X$-invariant smooth point $P$ of the threefold $X$ is a non-canonical center of  the log pair $(X,\lambda\mathcal{M})$,
then $\mathrm{mult}_{P}(\mathcal{M})>\frac{2}{\lambda}$.
\end{itemize}
In the remaining part of this section, we will prove that $\bigstar$ holds.

Let $\mathcal{M}$ be a non-empty $G$-invariant mobile  linear system on $X$,
and let $\lambda$ be a positive rational number such that $\lambda\mathcal{M}\sim_{\mathbb{Q}} -K_X$.

\begin{lemma}
\label{lemma:T-invariant-points}
Let $P$ be a smooth $\mathbb{T}$-invariant point of the toric Fano threefold $X$ such that $P$ is a non-canonical center of the log pair $(X,\lambda\mathcal{M})$.
Then $\mathrm{mult}_{P}(\mathcal{M})>\frac{2}{\lambda}$.
\end{lemma}

\begin{proof}
Since $P$ is a $\mathbb{T}$-invariant smooth point of $X$, we have $X=V_6$, $X=Y_{24}$~or~$X=\mathbb{P}^3$.
Let $G_P$ be the stabilizer of the point $P$ in $G$.
Then the tangent space $T_{P}(X)$ is a faithful representation of the group $G_P$ by Lemma~\ref{lemma:stabilizer-faithful}.
Moreover,  this representation is irreducible,
because $G_P$ contains $\Gamma$.
Thus, the assertion follows from Lemma~\ref{lemma:mult-2}.
\end{proof}

Now we suppose that $(X,\lambda\mathcal{M})$ is not canonical and we let $Z$ be a non-canonical $G$-center. To complete the proof, we have to show that $Z$ is $\mathbb{T}$-invariant. In what follows, we denote by  $I_X$ the~Fano index of the threefold $X$.

\begin{lemma}
\label{lemma:T-invariant-curves}
If $Z$ is a curve, then $Z$ is $\mathbb{T}$-invariant.
\end{lemma}

\begin{proof}
Let $H$ be an ample Cartier divisor on $X$ such that $-K_X\sim I_XH$. Then
$$
\frac{H^3I_X^2}{\lambda^2}=H\cdot M_1\cdot M_2\geqslant\big(H\cdot Z\big)\mathrm{mult}^2_{Z}\big(\mathcal{M}\big)>\frac{H\cdot Z}{\lambda^2}
$$
for two general surfaces $M_1$ and $M_2$ in $\mathcal{M}$, because $\mathrm{mult}_{Z}(\mathcal{M})>\frac{1}{\lambda}$ by Lemma~\ref{lemma:exercise}.
Then
$$
H\cdot Z<H^3I_X^2\leqslant 24.
$$
But $H$ is very ample, and the group $\Gamma$ contains a subgroup isomorphic to $\mumu_n^3$ for $n\geq 24$.
Thus, the curve $Z$ must be $\mathbb{T}$-invariant by Lemma~\ref{corollary:toric-degree}.
\end{proof}

Thus, we may assume that $Z$ is the $G$-orbit of a point in $X$.

\begin{lemma}
\label{lemma:points-in-T-invariant-curves}
The $G$-orbit $Z$ is contained in the union of $\mathbb{T}$-invariant curves.
\end{lemma}

\begin{proof}
Suppose that $Z$ is not contained in the union of $\mathbb{T}$-invariant curves.
Let us seek for a contradiction. Observe that the $G$-orbit $Z$ is a $G$-center of non-log canonical singularities of the log pair $(X,2\lambda\mathcal{M})$. We claim that $Z$ is an isolated center of non-log canonical singularities of this log pair.
Indeed, suppose that there is a $G$-irreducible curve $C$ that is a center of non-log canonical
singularities of the log pair $(X,2\lambda\mathcal{M})$.
Let $M_1$ and $M_2$ be  general surfaces in~$\mathcal{M}$. Then
$$
M_1\cdot M_2=mC+\Omega,
$$
where $m$ is a non-negative integer, and $\Omega$ is an effective one-cycle whose support does not contain $C$.
Then $m>\frac{1}{\lambda^2}$ by Lemma~\ref{lemma:Corti-surface}.

Let $H$ be an ample Cartier divisor on $X$ such that $-K_X\sim I_XH$. Then
$$
\frac{H^3I_X^2}{\lambda^2}=H\cdot M_1\cdot M_2=mH\cdot C+H\cdot\Omega\geqslant mH\cdot C>\frac{H\cdot C}{\lambda^2},
$$
so that $H\cdot C<H^3I_X^2\leqslant 24$. Thus, the curve $C$ must be $\mathbb{T}_X$-invariant by Lemma~\ref{corollary:toric-degree}.
Since $Z$ is not contained in the union of $\mathbb{T}$-invariant curves,
we see that $Z$ is an isolated center of non-log canonical singularities of the log pair $(X,2\lambda\mathcal{M})$.

Let $\mu$ be a positive rational number such that $\mu<\lambda$,
and $Z$ is an isolated $G$-irreducible center of log canonical singularities of the log pair $(X,2\mu\mathcal{M})$.
Let $\mathcal{I}$ be the multiplier ideal sheaf of the log pair~$(X,2\mu\mathcal{M})$.
Then the ideal $\mathcal{I}$ defines a subscheme $Z'$ in $X$
whose support contains $Z$.
Using Nadel vanishing theorem (see \cite[Theorem~9.4.8]{Lazarsfeld}), we get
$$
h^1\Big(\mathcal{I}\otimes\mathcal{O}_{X}\big(-K_{X}\big)\Big)=0.
$$
Now using the exact sequence of sheaves
$$
0\longrightarrow\mathcal{I}\otimes\mathcal{O}_{X}\big(-K_{X}\big)\longrightarrow\mathcal{O}_{X}\big(-K_{X}\big)\longrightarrow\mathcal{O}_{Z'}\otimes\mathcal{O}_{X}\big(-K_{X}\big)\longrightarrow 0,
$$
we see that $|Z|\leqslant h^0(\mathcal{O}_{X}(-K_{X}))$. But on the other hand, since  $Z$ is not contained in the union of $\mathbb{T}$-invariant curves, we have $|Z|\geqslant n^2\geqslant 24^2$, a contradiction.
\end{proof}

Finally, we prove the following lemma:

\begin{lemma}
\label{lemma:points-are-T-invariant}
The $G$-orbit $Z$ is $\mathbb{T}$-invariant.
\end{lemma}

\begin{proof}
We know from Lemma~\ref{lemma:points-in-T-invariant-curves} that $Z$ is contained in union of $\mathbb{T}$-invariant curves.
Suppose that $Z$ is not $\mathbb{T}$-invariant. Let us seek for a contradiction.

Let $H$ be an ample Cartier divisor on $X$ such that $-K_X\sim I_XH$. Recall that $H$ is very ample,
$$
I_X=\left\{\aligned
&2\ \text{if}\ X=V_6,\\
&1\ \text{if}\ X=X_{24},\\
&1\ \text{if}\ X=Y_{24},\\
&2\ \text{if}\ X=V_4,\\
&4\ \text{if}\ X=\mathbb{P}^3.\\
\endaligned
\right.
$$
and
$$
H^3=\left\{\aligned
&6\ \text{if}\ X=V_6,\\
&24\ \text{if}\ X=X_{24},\\
&24\ \text{if}\ X=Y_{24},\\
&4\ \text{if}\ X=V_4,\\
&1\ \text{if}\ X=\mathbb{P}^3.\\
\endaligned
\right.
$$

Let $\mathcal{C}$ the union of all  $\mathbb{T}$-invariant curves.
Then $\mathcal{C}$ is $G$-invariant.
Moreover, it follows from Lemmas~\ref{lemma:Y24-G-Fano} and \ref{lemma:V6-G-Fano} that
$\mathbb{W}$ acts transitively on the set of  irreducible components of $\mathcal{C}$ except the following two cases:
\begin{enumerate}
\item $X=Y_{24}$ and $\mathbb{W}=\mathbb{W}_{1}^{\mathfrak{A}}$,
\item $X=Y_{24}$ and $\mathbb{W}=\overline{\mathbb{W}}_{1}^{\mathfrak{S}}$.
\end{enumerate}
In these two cases, the curve $\mathcal{C}$ splits as a union of two $G$-irreducible curves that are swapped by the group $\mathbb{W}_X\cong\mathfrak{S}_4\times\mumu_2$.
Thus, we let $\mathcal{C}_1$ be the $G$-irreducible component of the curve $\mathcal{C}$ that contains the $G$-orbit $Z$.
Then, because $Z$ is $G$-irreducible, the other $G$-irreducible component of the curve $\mathcal{C}$ (if any) does not contain $Z$,

Let $d=H\cdot\mathcal{C}_1$, and let $k$ be the number of irreducible components of the curve $\mathcal{C}_1$.
Then $k=d$ in each possible case. Moreover, we have
$$
d=\left\{\aligned
&12\ \text{if}\ X=V_6,\\
&12\ \text{if}\ X=X_{24},\\
&24\ \text{if}\ X=Y_{24},\ \text{$\mathbb{W}\ne\mathbb{W}_{1}^{\mathfrak{A}}$ and $\mathbb{W}\ne\overline{\mathbb{W}}_{1}^{\mathfrak{S}}$},\\
&12\ \text{if}\ X=Y_{24},\ \text{and either $\mathbb{W}=\mathbb{W}_{1}^{\mathfrak{A}}$ or $\mathbb{W}=\overline{\mathbb{W}}_{1}^{\mathfrak{S}}$},\\
&12\ \text{if}\ X=V_4,\\
&6\ \text{if}\ X=\mathbb{P}^3.\\
\endaligned
\right.
$$
Recall that $G\cap\mathbb{T}$ contains a subgroup isomorphic to $\mumu_n^3$ for $n\geqslant 24$.
Thus, since by assumption $Z$ is not $\mathbb{T}$--invariant, we have  $|Z|\geqslant kn\geqslant 24k=24d$.

Let $M_1$ and $M_2$ be general surfaces in $\mathcal{M}$.
If the curve $\mathcal{C}$ is $G$-irreducible, then $\mathcal{C}=\mathcal{C}_1$ and
$$
M_1\cdot M_2=m_1\mathcal{C}_1+\Delta,
$$
where $m_1$ is a non-negative integer, and $\Delta$ is an effective one-cycle whose support does not contain the curve $\mathcal{C}$.
Similarly, if $X=Y_{24}$ and either $\mathbb{W}=\mathbb{W}_{1}^{\mathfrak{A}}$ or $\mathbb{W}=\overline{\mathbb{W}}_{1}^{\mathfrak{S}}$,
then $\mathcal{C}=\mathcal{C}_1+\mathcal{C}_2$, where both $\mathcal{C}_1$ and $\mathcal{C}_2$ are $G$-irreducible curves,
so that we have
$$
M_1\cdot M_2=m_1\mathcal{C}_1+m_2\mathcal{C}_2+\Delta,
$$
where $m_1$ and $m_2$ are non-negative integers, and $\Delta$ is an effective one-cycle whose support does not contain the curves $\mathcal{C}_1$ and $\mathcal{C}_2$. In both cases, we have
$$
\frac{I_X^2H^3}{\lambda^2}=H\cdot M_1\cdot M_2\geqslant H\cdot\Big(m_1\mathcal{C}_1+\Delta\Big)=m_1d+H\cdot\Delta\geqslant m_1d,
$$
which implies that $m_1\leqslant\frac{I_X^2H^3}{d\lambda^2}<\frac{4}{\lambda^2}$.

Let $\mathcal{B}$ be the linear subsystem in $|lH|$ consisting of surfaces that contain $\mathcal{C}$,
where
$$
l=\left\{\aligned
&2\ \text{if}\ X=V_6,\\
&2\ \text{if}\ X=X_{24},\\
&2\ \text{if}\ X=Y_{24},\\
&3\ \text{if}\ X=V_4,\\
&3\ \text{if}\ X=\mathbb{P}^3.\\
\endaligned
\right.
$$
Then $\mathcal{C}$ is the base locus of the linear system $\mathcal{B}$.
Indeed, the generators of the linear system $\mathcal{B}$ are
contained in the formulas \eqref{eq:Embed-X24-P13}, \eqref{equation:Y24-V6-P7}, \eqref{eq:Bir-V4-X24}, \eqref{equation:toric-quartic-double-solid-inverse} and \eqref{equation:Cremona}.
Looking at them, we see that the curve $\mathcal{C}$ is the base locus of the linear system $\mathcal{B}$.

Now we use Lemma~\ref{lemma:Pukhlikov} to deduce that
$$
\mathrm{mult}_{O}\Big(M_1\cdot M_2\Big)>\frac{4}{\lambda^2}
$$
for every point $O\in Z\cap\mathcal{C}$.
Thus, for every point $O\in Z\cap\mathcal{C}$, we have
$$
\mathrm{mult}_{O}\big(\Delta\big)>\frac{4}{\lambda^2}-m_1,
$$
because the curve $\mathcal{C}_1$ is smooth at $O$, and $\mathcal{C}_1$ is the only $G$-irreducible component of the curve $\mathcal{C}$ that
contains points in $Z$.
Now let $S$ be a general surface in $\mathcal{B}$. Then
$$
\frac{lI_X^2H^3}{\lambda^2}-m_1ld= S\cdot\Big(M_1\cdot M_2-m_1\mathcal{C}_1\Big)\geqslant S\cdot\Delta\geqslant\sum_{O\in Z}\mathrm{mult}_O\big(\Delta\big)>|Z|\Big(\frac{4}{\lambda^2}-m_1\Big)\geqslant dn\Big(\frac{4}{\lambda^2}-m_1\Big).
$$
This gives
$$
\frac{lI_X^2H^3}{\lambda^2}+m_1d\big(n-l\big)>\frac{4dn}{\lambda^2}.
$$
Since $l\leqslant n$ and $m_1\leqslant\frac{I_X^2H^3}{d\lambda^2}$, we obtain $I_X^2H^3>4d$, which is absurd.
The obtained contradiction completes the proof of the lemma.
\end{proof}

\appendix

\begin{landscape}

\section{Table of $G$-solid toric Fano threefolds}
\label{section:table}

With the notation of Section~\ref{section:intro}, we let $X$ be one of the toric threefolds among $V_6$, $V_{4}$, $X_{24}$, $Y_{24}$ and $\mathbb{P}^3$ and we let $G$ be a subgroup in the group $G_X$ whose image $\nu_X(G)$ in the Weyl group $\mathbb{W}_X$ of $X$  contains the subgroup $\mathfrak{A}_4$.
Then $X$ is $G$-minimal except the following two cases:
\begin{enumerate}
\item $X=V_4$, $\nu_X(G)\cong\mathfrak{S}_4$ and $G$ acts intransitively on $\mathbb{T}$-invariant surfaces,
\item $X=V_4$ and $\nu_X(G)\cong\mathfrak{A}_4$.
\end{enumerate}
Moreover, if $X$ is $G$-minimal and $|G|\geqslant 32\cdot 24^4$, then $X$ is $G$-solid by Theorem~\ref{theorem:main}.
In this case, the following table summarizes additional information on the $G$-equivariant birational geometry of the threefold $X$
obtained in the proof of Theorem~\ref{theorem:main}.

\medskip

\begin{center}
\renewcommand{\arraystretch}{3}
\begin{tabular}{|c||c|c|c|c|c|}
  \hline
$\nu_X(G)$ & $\mathfrak{S}_4\times\mumu_2$ & $\mathfrak{S}_4$ (type I) & $\mathfrak{S}_4$ (type II) & $\mathfrak{A}_4\times\mumu_2$ & $\mathfrak{A}_4$ \\
  \hline
  \hline
$V_6$ & \shortstack{$G$-birationally \\ superrigid} & \shortstack{$G$-birationally \\ superrigid} & \shortstack{$G$-birationally \\ superrigid} & \shortstack{$G$-birationally \\ superrigid} & \shortstack{$G$-birationally \\ superrigid} \\
  \hline
$V_4$ & $G$-birational to $X_{24}$ & not $G$-minimal & $G$-birational to $X_{24}$  & $G$-birational to $X_{24}$  & not $G$-minimal \\
  \hline
$X_{24}$ & $G$-birational to $V_{4}$ & $G$-birational to $\mathbb{P}^3$ & $G$-birational to $V_{4}$ & $G$-birational to $V_{4}$ & $G$-birational to $\mathbb{P}^3$\\
  \hline
$Y_{24}$ & \shortstack{$G$-birationally \\ superrigid} & \shortstack{$G$-birationally \\ superrigid} & \shortstack{$G$-birationally \\ superrigid} & \shortstack{$G$-birationally \\ superrigid} & \shortstack{$G$-birationally \\ superrigid} \\
  \hline
$\mathbb{P}^3$ & & $G$-birational to $X_{24}$ & & $G$-birational to $X_{24}$ & $G$-birational to $X_{24}$ \\
  \hline
\end{tabular}
\end{center}

\medskip

If $X\ne\mathbb{P}^3$, then $\mathbb{W}_X\cong\mathfrak{S}_4\times\mumu_2$ contains two subgroups isomorphic to $\mathfrak{S}_4$,
which we call the subgroups $\mathfrak{S}_4$ of type I~and~II, see Notation \ref{notation:subgroups} for the precise definition.
If~$\nu_{V_4}(G)$ is the subgroup $\mathfrak{S}_4$ of type II, then $G$ acts transitively on the set of irreducible $\mathbb{T}$-invariant surfaces,
so that $V_4$ is $G$-minimal in this case. In contrast, if $\nu_{V_4}(G)$ is the subgroup $\mathfrak{S}_4$ of type I, then $V_4$ is not $G$-minimal.

Similarly, if $\nu_{X_{24}}(G)$ is the subgroup $\mathfrak{S}_4$ of type II, then $G$ acts transitively on the set of singular points of the threefold $X_{24}$.
On the other hand, if $\nu_{X_{24}}(G)$ is the subgroup $\mathfrak{S}_4$ of type I, then $G$ does not acts transitively on this set.

\end{landscape}

\end{document}